\documentclass{amsart}
\usepackage{amsrefs}
\usepackage{mathtools}
\usepackage{hyperref}
\usepackage{defs}
\usepackage{yoshizumi_macros}

\newtheorem{theorem}{Theorem}[section]
\newtheorem{proposition}[theorem]{Proposition}
\newtheorem{lemma}[theorem]{Lemma}
\newtheorem{corollary}[theorem]{Corollary}
\newtheorem{conjecture}[theorem]{Conjecture}

\newtheorem{heuristic}[theorem]{Heuristic}

\theoremstyle{definition}
\newtheorem{definition}[theorem]{Definition}

\theoremstyle{remark}
\newtheorem{remark}[theorem]{Remark}

\numberwithin{equation}{section}

\begin{document}

\title[Geometric modular Polynomials with Level Structures]{Geometric construction of modular polynomials with level structures}


\author[H. Onuki]{Hiroshi Onuki}
\address{The University of Tokyo, Japan}
\email{hiroshi-onuki@g.ecc.u-tokyo.ac.jp}

\author[Y. Uchida]{Yukihiro Uchida}
\address{Tokyo Metropolitan University, Japan}
\email{yuchida@tmu.ac.jp}

\author[R. Yoshizumi]{Ryo Yoshizumi}
\address{Kyushu University, Japan}
\email{yoshizumi.ryo.483@s.kyushu-u.ac.jp}



\date{}

\dedicatory{}
\thanks{
    Authors are listed in alphabetical order;
    see \url{https://www.ams.org/profession/leaders/culture/JointResearchandItsPublicationfinal.pdf}.
    The authors would like to thank Yuya Murakami for
    helpful correspondence regarding an issue in the proof of \cite[Proposition 2.12]{Murakami2020},
    and for confirming that the argument can be corrected.
    The first author was supported by JST K Program Grant Number JPMJKP24U2, Japan.
    The second author was supported by JSPS KAKENHI Grant Number JP24K21512.
    The third author was supported by WISE program (MEXT) at Kyushu University.
}

\begin{abstract}
The classical modular polynomial for $j$-invariants
describes the relation between two elliptic curves connected by isogenies.
This polynomial has been applied to various algorithms in computational number theory,
such as point counting on elliptic curves.
In addition, computing the modular polynomial itself is also an important problem,
and various algorithms to compute it have been proposed.
On the other hand,
modular polynomials for other invariants of higher level structures
have also been studied.
For example, the modular polynomials for the Legendre $\lambda$-invariant
and the Weber functions are well-known.
In this paper, we give another approach to construct modular polynomials of higher level
purely algebraically.
In particular, we show the existence of modular polynomials
for invariants directly related to models of elliptic curves,
such as the coefficients of Montgomery and Hessian curves.
We also show that these modular polynomials have integer coefficients
and are symmetric and irreducible in certain cases,
and give an algorithm to compute them, which is based on the deformation method
by Kunzweiler and Robert.
\end{abstract}

\maketitle

\section{Introduction}
For a positive integer $m$,
the \emph{classical modular polynomial} $\Phi_m(X, Y) \in \Z[X, Y]$
is a polynomial that describes the relation between two $m$-isogenous $j$-invariants.
More precisely,
for $\tau$ in the upper half plane $\UH$ of the complex number field,
we denote by $E_\tau$ the elliptic curve over $\Comp$ defined by
$\Comp / (\Z + \Z\tau)$,
and by $j(\tau)$ the $j$-invariant of $E_\tau$.
Then, we have
\begin{equation*}
    \Phi_m(j(\tau), j(\tau')) = 0
\end{equation*}
if and only if there exists a cyclic isogeny of degree $m$
between the elliptic curves $E_\tau$ and $E_{\tau'}$.
The polynomial $\Phi_m(X, Y)$ is symmetric,
i.e., $\Phi_m(X, Y) = \Phi_m(Y, X)$,
and monic of degree $\psi(m)$ in both $X$ and $Y$,
where $\psi$ is the Dedekind function, i.e.,
\begin{equation*}
    \psi(m) = m \prod_{p \mid m} \left(1 + \frac{1}{p}\right)
\end{equation*}
with the product over all prime divisors $p$ of $m$.

\subsection{Algorithms for computing modular polynomials}
Computing the classical modular polynomials is a well-studied problem.
The algorithms to compute them include
a method using Fourier expansion \cite{Elk97}
and a method using isogenies between supersingular elliptic curves
with the Chinese Remainder Theorem (CRT)~\cite{CL05}.
Enge \cite{Eng09} provides an algorithm in $O(\l^3(\log\l)^{4+\epsilon})$ using interpolation and floating-point evaluation 
under some heuristics.
However, this algorithm needs an excessive amount of memory for large $\l$. 
Then, Br\"{o}ker, Lauter, and Sutherland~\cite{broker2012modular} provided
an algorithm in $O(\l^3\log^3\l\log\log \l)$ using
isogeny volcanoes with the CRT method under Generalized Riemann Hypothesis (GRH),
which requires a practical amount of memory.
Recently, Kunzweiler and Robert~\cite{deformation_Kunzweiler_2024} proposed
an algorithm in the same complexity without GRH.
Their algorithm also uses the CRT method,
and is based on the deformation theory and
a higher dimensional representation of isogenies between elliptic curves
based on Kani's work~\cite{Kan1997number},
developed by \cite{EC:CasDec23,EC:MMPPW23,EC:Robert23}.

\subsection{Modular polynomials of higher level}
Another characterization of the classical modular polynomial
is that $\Phi_m(X, j(\tau))$ is the minimal polynomial of $j(m\tau)$
over $\Q(j(\tau))$ for $\tau \in \UH$ such that $j(\tau)$ is transcendental over $\Q$.
Note that $E_{m\tau}$ is the codomain of the isogeny defined by
$z \mapsto mz$,
whose kernel is the cyclic subgroup of order $m$ generated by
$\frac{1}{m}$ in $E_\tau$.

This characterization allows us to define \emph{modular polynomials of higher level} as follows:
Let $f$ be a modular function for a congruence subgroup $\Gamma$ of level $N$ in $\SL$,
and let $\tau$ be an element in $\UH$ such that $f(\tau)$ is a transcendental number.
Let $m$ be a positive integer prime to $N$,
and $\Phi'(X)$ be the minimal polynomial of $f(m\tau)$ over $\Q(f(\tau))$.
Then, we define the \emph{modular polynomial of order $m$ for $f$}
to be the polynomial $\Phi_m^f(X, Y) \in \Q[X, Y]$ obtained by
replacing $f(\tau)$ by $Y$ in $\Phi'(X)$,
and denote it by $\Phi_m^f(X, Y)$.
For example,
it is known that there exist the modular polynomials
for the Weber function and the Legendre $\lambda$-invariant
(see \cite{Broker2011,broker2012modular} and \cite[Chap. 4.4]{PiAGM} for example).
As mentioned in \cite{broker2012modular},
if the Fourier expansion of $f$ has integer coefficients
then $\Phi_m^f(X, Y)$ has integer coefficients.
In addition,
the modular polynomial $\Phi_\ell^f(X, Y)$ for a prime $\ell \nmid N$
is symmetric if $f$ is invariant under the action of
certain matrices~\cite[Lemma 5.5]{Broker2011}.

\subsection{Geometric viewpoint of modular polynomials of higher level}
From geometric viewpoints,
the modular polynomial $\Phi_m^f(X, Y)$
describes the relation between isomorphism classes of enhanced elliptic curves
of level $\Gamma$ connected by isogenies of degree $m$
which preserve the level structure.
By observing the images of $\tau/N$ and $1/N$
under the isogeny $E_\tau \to E_{m\tau}$,
we can define \emph{isogenies between enhanced elliptic curves} as follows:
\begin{itemize}
    \item An \emph{enhanced elliptic curve of level $\Gamma(N)$}
        is a pair $(E, (P, Q))$ of an elliptic curve $E$ over $\Comp$
        and a basis $(P, Q)$ of the $N$-torsion subgroup $E[N]$ of $E$
        such that $e_N(P, Q) = e^{2\pi i / N}$,
        where $e_N$ is the Weil pairing on $E[N]$.
        An \emph{isogeny of degree $m$ between enhanced elliptic curves}
        $(E, (P, Q))$ and $(E', (P', Q'))$ of level $\Gamma(N)$
        is an isogeny $\varphi: E \to E'$ of degree $m$
        such that $\varphi(P) = [m]P'$ and $\varphi(Q) = Q'$.
    \item An \emph{enhanced elliptic curve of level $\Gamma_1(N)$}
        is a pair $(E, P)$ of an elliptic curve $E$ over $\Comp$
        and a point $P \in E[N]$ of order $N$.
        An \emph{isogeny of degree $m$ between enhanced elliptic curves}
        $(E, P)$ and $(E', P')$ of level $\Gamma_1(N)$
        is an isogeny $\varphi: E \to E'$ of degree $m$
        such that $\varphi(P) = P'$.
    \item An \emph{enhanced elliptic curve of level $\Gamma_0(N)$}
        is a pair $(E, C)$ of an elliptic curve $E$ over $\Comp$
        and a cyclic subgroup $C$ of $E[N]$ of order $N$.
        An \emph{isogeny of degree $m$ between enhanced elliptic curves}
        $(E, C)$ and $(E', C')$ of level $\Gamma_0(N)$
        is an isogeny $\varphi: E \to E'$ of degree $m$
        such that $\varphi(C) = C'$.
\end{itemize}
To simplify the notation,
we denote by $(E, \ast)$ an enhanced elliptic curve of level $\Gamma$,
where $\Gamma$ is one of $\Gamma(N)$, $\Gamma_1(N)$, and $\Gamma_0(N)$
if we do not need to specify the level structure.

The main target of this paper is \emph{(geometric) modular polynomials} of higher level,
which we define as follows:
\begin{definition}\label{def:modular_polynomial_higher_level}
    Let $\Gamma$ be $\Gamma(N)$, $\Gamma_1(N)$, or $\Gamma_0(N)$,
    let $f$ be an embedding from the set of isomorphism classes of
    enhanced elliptic curves of level $\Gamma$ to $\Comp$,
    and $m$ be a positive integer prime to $N$.
    A \emph{modular polynomial of order $m$ for $f$}
    is the polynomial $\Phi_m^f(X, Y) \in \Comp[X, Y]$
    such that
    \begin{itemize}
        \item $\Phi_m^f(f(E, \ast), f(E', \ast)) = 0$
            if and only if there exists a cyclic isogeny of degree $m$
            between the enhanced elliptic curves $(E, \ast)$ and $(E', \ast)$,
        \item $\Phi_m^f(X, Y)$ is monic of degree $\psi(m)$ in $Y$.
    \end{itemize}
\end{definition}
Note that a polynomial satisfying the first condition has degree at least $\psi(m)$ in $Y$,
since the number of cyclic subgroups of order $m$ in an elliptic curve is equal to $\psi(m)$.
Therefore, the above two conditions imply the uniqueness of modular polynomials.

This characterization of modular polynomials can be found in \cite{Murakami2020}.
That paper proves that 
if $f$ is a Hauptmodul of the canonical compactification $X_0(N)$ of $\Gamma_0(N)\backslash \UH$
such that the Fourier expansion of $f$ has integer coefficients
and its leading coefficient is 1,
then the modular polynomial $\Phi_m^f(X, Y)$ has integer coefficients,
is symmetric\footnote{
    The proof in \cite[Proposition 2.12]{Murakami2020} is incorrect as stated.
    However, the author has confirmed to us that it can be fixed.
}, and monic of degree $\psi(m)$ in both $X$ and $Y$
for a positive integer $m$ prime to $N$
(Theorem 2.9 and Proposition 2.12 in \cite{Murakami2020}).
We note that the proof in \cite{Murakami2020} can be easily extended to the case of $\Gamma_1(N)$.

\subsection{Contributions}
The Legendre $\lambda$-invariant has the elliptic curve model
\begin{equation}\label{eq:legendre_form}
    E_\lambda: Y^2Z = X(X - Z)(X - \lambda Z),
\end{equation}
and its basis $((0: 0: 1), (1: 0: 1))$ of the $2$-torsion subgroup.
Therefore, we can regard the modular polynomial for the Legendre $\lambda$-invariant
as the modular polynomial of higher level in Definition \ref{def:modular_polynomial_higher_level}.
On the other hand,
there are other forms of elliptic curves
which can be regarded as enhanced elliptic curves of higher level.
For example,
a Montgomery curve
is known to be an enhanced elliptic curve of level $\Gamma_0(4)$
(we will explain this in Section~\ref{subsec:montgomery}).
Therefore, it is natural to expect the existence of modular polynomials
for the coefficients of Montgomery curves.
However, to the best of our knowledge,
there is no literature proving the existence of modular polynomials
for such forms of elliptic curves.

One possible approach to prove the existence of modular polynomials
is to find an expression of the coefficients of Montgomery curves
as a modular function for $\Gamma_0(4)$,
and to show that its Fourier expansion has integer coefficients.

In this paper, we take another approach to construct modular polynomials of higher level.
In particular,
we show the existence of modular polynomials for a map $f$ in Definition \ref{def:modular_polynomial_higher_level}
when $f$ has scalar multiplication formulas, division polynomials, and isogeny formulas
satisfying certain conditions.
These conditions are designed to make $f(E', \ast)$ integral
over $\Z[f(E, \ast)]$, where $(E, \ast)$ and $(E', \ast)$ are enhanced elliptic curves
connected by a cyclic isogeny.
Under these conditions,
we prove that the modular polynomial $\Phi_m^f(X, Y)$
exists and has integer coefficients for any positive integer $m$ prime to $N$.
In addition,
we show that $\Phi_m^f(X, Y)$ is symmetric
if $m > 1$, and the level of $f$ is $\Gamma_0(N)$ or $m \equiv \pm 1 \pmod{N}$
when the level of $f$ is $\Gamma_1(N)$ or $\Gamma(N)$.
Examples of such maps $f$ include
the coefficients of Montgomery and Hessian curves.
We also give an algorithm to compute such modular polynomials,
which is based on the algorithm in \cite{deformation_Kunzweiler_2024}.

We do not claim that our examples are not modular functions.
Rather, we believe that they are modular functions.
Our contribution is to show another approach to construct modular polynomials,
which is purely algebraic and without using the theory of modular forms.
We believe that our approach could be useful to construct modular polynomials
for other forms of elliptic curves.

\section{Invariants and models for enhanced elliptic curves}\label{sec:invariants_and_models}
In this section, we introduce the notion of
an \emph{invariant of level $\Gamma$} and its \emph{good model} for a congruence subgroup $\Gamma$,
which are necessary to prove the existence of modular polynomials of higher level.
\begin{definition}
    Let $\Gamma$ be $\Gamma(N)$, $\Gamma_1(N)$, or $\Gamma_0(N)$,
    and $S$ be the set of isomorphism classes of
    enhanced elliptic curves of level $\Gamma$.
    An \emph{invariant of level $\Gamma$} is an injection
    \begin{equation*}
        \alpha : S \to \Comp
    \end{equation*}
    accompanied by a rational function $j_\alpha(X) \in \Q(X)$
    satisfying the following conditions:
    \begin{itemize}
        \item For a representative $(E, *)$ of an enhanced elliptic curve in $S$,
            the $j$-invariant of $E$ is given by $j(E) = j_\alpha(\alpha(E, *))$.
        \item Let $J_1(X)/J_0(X)$ be the reduced form of $j_\alpha(X)$
            with $J_0(X), J_1(X) \in \Z[X]$.
            Then, the following conditions hold:
            \begin{itemize}
                \item $\deg J_1(X) > \deg J_0(X)$,
                \item $J_0(z) = 0$ for all $z \in \Comp \setminus \image \alpha$,
                \item $J_0(X)$ is primitive,
                \item the leading coefficient of $J_1(X)$ is not divisible by any prime not dividing $N$.
            \end{itemize}
    \end{itemize}
\end{definition}

\begin{remark}\label{rem:degree_J1}
    Although we do not restrict the degree of $J_1(X)$,
    it should be equal to the index $[\SL : \Gamma]$
    since the number of isomorphism classes of enhanced elliptic curves of level $\Gamma$
    with a fixed $j$-invariant is equal to $[\SL : \Gamma]$.
\end{remark}

Obviously, the $j$-invariant itself is an invariant of level $\SL$.
For another example,
it is known that the Legendre $\lambda$-invariant
is an invariant of level $\Gamma(2)$ (see \cite[Chapter 18, \S 6]{lang1987elliptic}).
Since the $j$-invariant of the Legendre form \eqref{eq:legendre_form} is given by
\begin{equation*}
    j(\lambda) = \frac{2^8(\lambda^2 - \lambda + 1)^3}{\lambda^2 (\lambda - 1)^2}, 
\end{equation*}
the polynomials $J_1(X) = 2^8 (X^2 - X + 1)^3$ and $J_0(X) = X^2 (X - 1)^2$
satisfy the above conditions for $N = 2$.
In addition,
as we will see in Section~\ref{sec:examples},
the coefficients of Montgomery and Hessian curves are invariants of level $\Gamma_0(4)$ and $\Gamma(3)$, respectively.

In addition to the above conditions,
we require a complete set of representatives of enhanced elliptic curves
with coordinates having formulas for scalar multiplications, division polynomials, and isogenies.
We call such a set a \emph{good model} of the invariant.
Intuitively,
for a certain form of elliptic curves,
an invariant is the coefficient of an elliptic curve in that form,
and a model of the invariant is a pair of the form of elliptic curves itself
and the $x$-coordinate of the form.
More precisely, we define a \emph{model of an invariant of level $\Gamma$} as follows:
\begin{definition}
    Let $\Gamma$ be $\Gamma(N)$, $\Gamma_1(N)$, or $\Gamma_0(N)$,
    and $\alpha$ be an invariant of level $\Gamma$.
    A \emph{model} of $\alpha$ is
    a complete set of representatives
    $\{(E_z^\alpha, \ast) \mid z \in \image \alpha\}$
    of enhanced elliptic curves of level $\Gamma$
    accompanied by
    a set of functions
    $\{\coord{z}{\alpha} : E_z^\alpha(\Comp) \to \Comp \cup \{\infty\} \mid z \in \image\alpha\}$
    satisfying the following conditions:
    \begin{enumerate}
        \item For any $z \in \image \alpha$,
            the function $\coord{z}{\alpha}$ is surjective. \label{item:model_surjective}
        \item For any $z \in \image \alpha$, we have
            $(\coord{z}{\alpha})^{-1}(\infty) \subset E_z^\alpha[N]$. \label{item:model_infinity}
        \item For any $R_1, R_2 \in E_z^\alpha(\Comp)$,
            we have $\coord{z}{\alpha}(R_1) = \coord{z}{\alpha}(R_2)$
            if and only if $R_1 = R_2 + R_\infty$ or $R_1 = -R_2 + R_\infty$
            for some $R_\infty \in (\coord{z}{\alpha})^{-1}(\infty)$. \label{item:model_equivalence}
    \end{enumerate}

    We say that a model $\{((E_z^\alpha, \ast), \coord{z}{\alpha})\}$
    is a \emph{good model} if
    there exist sets of rational functions and polynomials,
    which we call the \emph{scalar multiplication formulas},
    \emph{division polynomials}, and \emph{isogeny formulas},
    defined as follows:
    \begin{itemize}
        \item (\textbf{Scalar multiplication formula})
            For an integer $m$,
            a \emph{scalar multiplication formula by $m$}
            is a rational function $M_m(X; Y) \in \Z(X; Y)$
            satisfying
            \begin{equation*}
                \coord{z}{\alpha}([m]R) = M_m(z; \coord{z}{\alpha}(R))
            \end{equation*}
            for any $z \in \image \alpha$ and $R \in E_z^\alpha(\Comp)$.
        \item (\textbf{Division polynomials})
            For a positive integer $m$,
            a \emph{division polynomial of order $m$}
            is a polynomial
            $\Psi_m(X; Y) \in \Z[X; Y]$ satisfying the following conditions:
            \begin{itemize}
                \item The leading coefficient of $\Psi_m(X; Y)$ as a polynomial in $Y$ is an integer not divisible by any prime not dividing $m$.
                \item For any $z \in \image \alpha$ and $R \in E_z^\alpha(\Comp) \setminus (\coord{z}{\alpha})^{-1}(\infty)$,
                    \begin{equation*}
                        \Psi_m(z; \coord{z}{\alpha}(R)) = 0
                    \end{equation*}
                    if and only if $\coord{z}{\alpha}([m]R) = \coord{z}{\alpha}(0_{E_z^\alpha})$,
                    where $0_{E_z^\alpha}$ is the identity of $E_z^\alpha$.
            \end{itemize}
        \item (\textbf{Isogeny formula})
            For a positive integer $m$ prime to $N$,
            an \emph{isogeny formula of degree $m$}
            is a polynomial
            $V_m(X; Y_1, \dots, Y_{m-1}) \in \Z\left[\frac{1}{m}\right][X; Y_1, \dots, Y_{m-1}]$
            satisfying the following condition:
            For any $z, z' \in \image \alpha$ such that
            there exists
            a cyclic isogeny $\varphi: E_z^\alpha \to E_{z'}^\alpha$ of degree $m$,
            it holds that
            \begin{equation*}
                z' = V_m\left(z; \coord{z}{\alpha}(R), \coord{z}{\alpha}([2]R), \dots, \coord{z}{\alpha}([(m-1)]R)\right)
            \end{equation*}
            for any generator $R \in E_z^\alpha(\Comp)$ of $\ker \varphi$.
    \end{itemize}
\end{definition}

From the properties of a good model,
the function $\coord{z}{\alpha}$ induces a surjection from $E_z^\alpha[m]\setminus\{0_{E_z^\alpha}\}$ to
the set $\{\tau \in \Comp \mid \Psi_m(z; \tau) = 0 \text{ and } \tau \neq \coord{z}{\alpha}(0_{E_z^\alpha})\}$
for any $z \in \image \alpha$ and positive integer $m$ prime to $N$.
In addition, two points $R_1, R_2 \in E_z^\alpha[m]$ have the same image under $\coord{z}{\alpha}$
if and only if $R_1 = \pm R_2$.

It is obvious that the $j$-invariant has a model; for example,
we can take
\begin{equation*}
    E_j^\mathrm{jinv} : Y^2Z = X^3 - \frac{27j}{j - 1728} XZ^2 - \frac{54j}{j - 1728} Z^3
\end{equation*}
with the $x$-coordinate as a model of the $j$-invariant
except for $j = 0, 1728$.
However, the division polynomials and isogeny formulas
for this model are not expressed in polynomials with coefficients in $\Z[j]$.
Therefore, the above model is not a good model.
On the other hand,
the Montgomery form and Hessian form
give good models of the invariants defined by their coefficients
(see Section~\ref{sec:examples}).

We also define a property of invariants
necessary to prove the irreducibility of modular polynomials
and to construct an algorithm to compute them.
\begin{definition}\label{def:computable}
    Let $\Gamma$ be $\Gamma(N)$, $\Gamma_1(N)$, or $\Gamma_0(N)$,
    and $\alpha$ be an invariant of level $\Gamma$.
    We say that $\alpha$ is \emph{computable} if,
    for any enhanced elliptic curve $(E^\mathrm{jinv}_j, S)$ of level $\Gamma$
    with $j$-invariant $j \in \Comp\setminus\{0, 1728\}$,
    the invariant $\alpha(E^\mathrm{jinv}_j, S)$ is represented by
    a rational function in $j$
    and the coordinates of the points in $S$,
    and the rational function does not depend on the choice of $j$ and $S$.
\end{definition}
Since our definition of invariants is motivated by examples such as
the coefficient of a Montgomery curve or a Hessian curve,
\Cref{def:computable} is a natural condition.
Indeed, these examples of invariants are computable
(see \Cref{prop:montgomery_well_defined} and \Cref{lem:isom_E_to_Hess}).

\section{The existence and properties of modular polynomials}\label{sec:main_theorems}
In this section, we state and prove the main theorems of this paper.
First, we show the existence of modular polynomials
for an invariant with a good model.
Next, we show that such modular polynomials have integer coefficients,
and are symmetric and irreducible in certain cases.
Finally, we give some conjectures for further properties of modular polynomials.

Throughout the rest of this paper,
we let $\Gamma$ be $\Gamma(N)$, $\Gamma_1(N)$, or $\Gamma_0(N)$
for a positive integer $N$,
and $\alpha$ be an invariant of level $\Gamma$ with a good model
$\{((E_z^\alpha, *), \coord{z}{\alpha})\}$.
In addition,
we let $J_1(X)/J_0(X)$ be the reduced form of $j_\alpha(X)$,
and $M_m(X; Y)$, $\Psi_m(X; Y)$, and $V_m(X; Y_1, \dots, Y_{m-1})$
be the scalar multiplication formulas,
division polynomials, and isogeny formulas of the good model
$\{((E_z^\alpha, *), \coord{z}{\alpha})\}$.
We also let $\Phi_m$ be the classical modular polynomial.

\subsection{Main theorems}\label{subsec:main_theorem}
First, we show the existence of modular polynomials
for an invariant with a good model.
\begin{theorem}\label{thm:existence_of_modular_polynomial}
    For any positive integer $m$ prime to $N$,
    there exists the modular polynomial $\Phi_m^\alpha(X, Y) \in \Q[X, Y]$
    of order $m$ for $\alpha$.
\end{theorem}

\begin{proof}
    Consider the univariate rational function field $\Q(X)$ over $\Q$.
    Let $K$ be the minimal splitting field of the polynomial
    $\Psi_m(X; Y)$ over $\Q(X)$,
    and $\tau$ be a primitive element of $K/\Q(X)$.
    By multiplying a suitable rational function in $\Q(X)$ to $\tau$,
    we may assume that all roots of $\Psi_m(X; Y)$ are contained in $\Q[X, \tau]$.
    Let $\tilde{f}(Y) \in \Q(X)[Y]$ be the minimal polynomial of $\tau$ over $\Q(X)$,
    and $f(Y) \in \Q[X][Y]$ be a polynomial obtained by multiplying
    a suitable polynomial in $\Q[X]$ to $\tilde{f}(Y)$.
    
    For $z \in \image \alpha$,
    we denote by $f_z(Y)$ the polynomial in $\Q[z][Y]$
    obtained by replacing $X$ by $z$ in the coefficients of $f(Y)$.
    Let $\tau_z$ be a root of $f_z(Y)$ in $\Comp$.
    Then we have a ring homomorphism
    $\rho_z : \Q[X, \tau] \to \Comp$
    defined by $X$ mapping to $z$ and $\tau$ mapping to $\tau_z$.

    Let $t \in \image \alpha$ be transcendental over $\Q$.
    Then $\rho_t$ induces an isomorphism onto $\Q[t, \tau_t]$.
    Let $\setR$ be $\rho_t^{-1}(\{\coord{t}{\alpha}(R) \mid R \in E_t^\alpha[m]\setminus\{0_{E_t^\alpha}\}\})$.
    Then $\setR$ is contained in the set of roots of $\Psi_m(X; Y)$ in $K$.
    By the choice of $\tau$,
    we have $\setR \subset \Q[X, \tau]$.
    As we mentioned in Section~\ref{sec:invariants_and_models},
    we can identify the set $\setR$ with the quotient
    $(E_t^\alpha[m]\setminus\{0_{E_t^\alpha}\})/\{\pm 1\}$
    via the isomorphism $\rho_t$.
    Therefore, we can take the subsets of $\setR$
    corresponding to a cyclic subgroup of order $m$ in $E_t^\alpha$.
    Let $C_1, \dots, C_{\psi(m)}$ be all such subsets of $\setR$.
    Then we show below that
    $\rho_{z}(C_1), \dots, \rho_{z}(C_{\psi(m)})$
    are the sets corresponding to
    the cyclic subgroups of order $m$ in $E_{z}^\alpha$
    for any $z \in \image \alpha$.

    Since $\Psi_m(X; Y)$ splits into linear factors over $\Q[X, \tau]$,
    the image $\Psi_m(z; Y)$ splits into linear factors over $\Q[z, \tau_z]$.
    In addition,
    the leading coefficient of $\Psi_m(z; Y)$ in $Y$ is in $\Q$.
    This means that all roots of $\Psi_m(z; Y)$ are contained in $\Q[z, \tau_z]$.
    Therefore, $\rho_{z}$ induces a surjection from $\setR$
    to the set of $\coord{z}{\alpha}(R)$ for $R \in E_z^\alpha[m]\setminus\{0_{E_z^\alpha}\}$.
    Since the cardinalities of both sets are equal to each other,
    the map $\rho_{z}$ induces a bijection between them.
    For any integer $i$ and $\mu \in \setR$,
    it holds that $M_i(X; \mu) = \infty$ implies $M_i(z; \rho_{z}(\mu)) = \infty$.
    Thus, the order of the point corresponding to $\rho_t(\mu)$ in $E_{t}^\alpha$
    and that of $\rho_{z}(\mu)$ are equal.
    In addition,
    it holds that $M_i(X; \mu) = \mu'$ implies
    $M_i(z; \rho_{z}(\mu)) = \rho_{z}(\mu')$
    for any integer $i$ and $\mu, \mu' \in \setR$.
    Therefore, the image $\rho_{z}(C_i)$
    is the set corresponding to
    a cyclic subgroup of order $m$ in $E_{z}^\alpha$.
    It follows that $\rho_{z}(C_i) \neq \rho_{z}(C_j)$ for $i \neq j$
    since $\rho_z$ induces a surjection from $\setR$ to
    the set of $\coord{z}{\alpha}(R)$ for $R \in E_z^\alpha[m]\setminus\{0_{E_z^\alpha}\}$.

    Let $\mu_i$ be an element of $C_i$ corresponding to a generator of the cyclic subgroup.
    We define a polynomial in $\Q[X, \tau][Y]$ by
    \begin{equation}\label{eq:def_mod_poly}
        F(Y) = \prod_{i=1}^{\psi(m)} (Y - \sigma_i),
    \end{equation}
    where
    $\sigma_i = V_m\left(X; \mu_i, M_2(X; \mu_i), \dots, M_{m-1}(X; \mu_i)\right) \in \Q[X, \tau]$.
    Then, for any $z \in \image \alpha$, the image $\rho_z(F(Y)) \in \Q[z, \tau_z][Y]$ satisfies,
    for any $z' \in \image \alpha$,
    $\rho_z(F(z')) = 0$
    if and only if there exists a cyclic isogeny of degree $m$
    between the enhanced elliptic curves
    $(E_z^\alpha, \ast)$ and $(E_{z'}^\alpha, \ast)$.
    Since $M_2, \dots, M_{m-1}$ and $V_m$ are invariant under
    the action of the Galois group of $K/\Q(X)$,
    any element in the Galois group acts on the set $\{\sigma_i \mid 1 \leq i \leq \psi(m)\}$
    as a permutation.
    Therefore, the coefficients of $F(Y)$ are contained in $\Q(X)$.
    In addition,
    since $\mu_i, M_2(X; \mu_i), \dots, M_{m-1}(X; \mu_i)$
    are roots of $\Psi_m(X; Y)$ whose leading coefficient in $Y$ is in $\Z$,
    the coefficients of $F(Y)$ are contained in $\Q[X]$.
    We define $\Phi_m^{\alpha}(X, Y) \in \Q[X, Y]$
    as $F(Y)$ viewed as a polynomial in $\Q[X, Y]$.
    By the construction,
    $\Phi_m^{\alpha}(X, Y)$ satisfies the conditions for the modular polynomial of order $m$ for $\alpha$.
\end{proof}

Next, we show that the modular polynomial constructed in Theorem \ref{thm:existence_of_modular_polynomial}
has integer coefficients and is of degree $\psi(m)$ in $X$.

\begin{theorem}\label{thm:integral_modular_polynomial}
    The modular polynomial $\Phi_m^{\alpha}(X, Y)$ satisfies the following:
    \begin{enumerate}
        \item $\Phi_m^{\alpha}(X, Y) \in \Z[X, Y]$. \label{enum:integer_coeff}
        \item $\deg_X(\Phi_m^{\alpha}) = \psi(m)$. \label{enum:degree_X}
    \end{enumerate}
\end{theorem}

\begin{proof}
    (\ref{enum:integer_coeff}):
    Recall the definition of $\Phi_m^{\alpha}(X, Y)$ in \eqref{eq:def_mod_poly}.
    Since the leading coefficient of $\Psi_m(X; Y)$ in $Y$ is an integer not divisible by any prime not dividing $m$,
    and $V_m$ has coefficients in $\Z\left[\frac{1}{m}\right]$,
    the roots $\sigma_1, \dots, \sigma_{\psi(m)}$ of $\Phi_m^{\alpha}(X, Y)$
    as a polynomial in $Y$
    are integral over $\Z\left[\frac{1}{m}, X\right]$.

    Let $j_1, \dots, j_{\psi(m)}$ be the roots of
    the classical modular polynomial $\Phi_m\left(\frac{J_1(X)}{J_0(X)}, Y\right)$
    in an algebraic closure of $\Q(X)$.
    Then, we have $j_i = J_1(\sigma_i)/J_0(\sigma_i)$ for all $1 \leq i \leq \psi(m)$
    by reindexing if necessary.
    Since $\Phi_m$ has integer coefficients and is monic in $Y$,
    the root $j_i$ is integral over $\Z\left[\frac{J_1(X)}{J_0(X)}\right]$ for all $1 \leq i \leq \psi(m)$.
    Let $c$ be the leading coefficient of $J_1(X)$.
    Then, $\sigma_i$ is integral over $\Z\left[\frac{1}{c}, j_i\right]$
    since we have $J_1(\sigma_i) - J_0(\sigma_i) j_i = 0$.
    Thus, 
    $\sigma_i$ is integral over $\Z\left[\frac{1}{c}, \frac{J_1(X)}{J_0(X)}\right]$.

    Combining the above results,
    we see that $\sigma_i$ is integral over $\Z\left[\frac{1}{m}, X\right] \cap \Z\left[\frac{1}{c}, \frac{J_1(X)}{J_0(X)}\right]$.
    Since $m$ is prime to $c$ and $J_0(X)$ is primitive,
    we have $\Z\left[\frac{1}{m}, X\right] \cap \Z\left[\frac{1}{c}, \frac{J_1(X)}{J_0(X)}\right] = \Z[X]$.
    This implies that the coefficients of $\Phi_m^{\alpha}(X, Y)$ are in $\Z$.

    (\ref{enum:degree_X}):
    Let $F(X, Y) = J_1(X) - J_0(X) Y$,
    and $d = \deg J_1$.
    Consider four variables $X, Y, Z, W$,
    and define a polynomial $G(X, Y) \in \Z[X, Y]$ by
    \begin{equation*}
        G(X, Y) = \res_Z(\res_W(F(W, X), \Phi_m^\alpha(W, Z)), F(Z, Y)) \in \Z[X, Y].
    \end{equation*}
    Here, $\res_Z$ and $\res_W$ denote the resultant with respect to
    variables $Z$ and $W$, respectively.
    By the property of resultant,
    for any $j, j' \in \Comp$,
    it holds that $G(j, j') = 0$
    if and only if there exist $z, z' \in \Comp$
    such that $F(z, j) = 0$, $F(z', j') = 0$,
    and $\Phi_m^\alpha(z, z') = 0$.
    If $z \not\in \image \alpha$ then $J_0(z) = 0$,
    and thus $F(z, j) \neq 0$ for any $j \in \Comp$.
    This means that $G(j, j') = 0$ if and only if
    there exist $z, z' \in \image \alpha$
    such that $E_z^\alpha$ and $E_{z'}^\alpha$ are connected by
    a cyclic isogeny of degree $m$,
    and $j = j(E_z^\alpha)$, $j' = j(E_{z'}^\alpha)$.
    Therefore, $G(j, j') = 0$ if and only if $\Phi_m(j, j') = 0$.
    From the irreducibility of $\Phi_m(X, Y)$,
    it follows that $G(X, Y) = \gamma \cdot \Phi_m(X, Y)^n$ for some $\gamma \in \Comp^\times$ and $n \in \Z_{\geq 1}$.

    For $j \in \Comp$, we let $\omega_1, \dots, \omega_d$
    be the roots of $F(W, j)$ in $\Comp$.
    For $i = 1, \dots, d$,
    let $z_{i, 1}, \dots, z_{i, \psi(m)}$ be the roots of $\Phi_m^\alpha(\omega_i, Z)$ in $\Comp$.
    Then, we have
    \begin{align*}
        G(j, Y) &= \res_Z\left(c^{\deg_X(\Phi_m^\alpha)}\prod_{i=1}^{d} \Phi_m^\alpha(\omega_i, Z), F(Z, Y)\right)\\
            &= c^{d\cdot\deg_X(\Phi_m^\alpha)} \prod_{i=1}^{d} \prod_{j=1}^{\psi(m)} F(z_{i, j}, Y).
    \end{align*}
    Therefore, we have $\deg_Y(G(j, Y)) = d \cdot \psi(m)$.
    Since this holds for any $j \in \Comp$,
    we have $\deg_Y(G(X, Y)) = d \cdot \psi(m)$.
    Thus, we have $G(X, Y) = \gamma \cdot \Phi_m(X, Y)^d$.

    Similar discussion considering $G(X, j)$ for $j \in \Comp$
    shows that $\deg_X(G(X, Y)) = d\cdot\deg_X(\Phi_m^\alpha)$.
    By comparing the degrees of $G(X, Y)$ and $\Phi_m(X, Y)^d$ in $X$,
    we obtain $\deg_X(\Phi_m^\alpha) = \psi(m)$.
\end{proof}

If $m > 1$, then the classical modular polynomial $\Phi_m(X, Y)$ is symmetric.
This property reflects the fact that the dual isogeny of a cyclic isogeny of degree $m$ is also
a cyclic isogeny of degree $m$.
However, this property does not always hold
for modular polynomials of $\Gamma(N)$ for $N > 4$.
Let $\varphi : (E, (P, Q)) \to (E', (P', Q'))$ be a cyclic isogeny of degree $m$
between enhanced elliptic curves of level $\Gamma(N)$.
Then, we have $\varphi(P) = [m]P'$ and $\varphi(Q) = Q'$
by the definition of isogenies between
enhanced elliptic curves of level $\Gamma(N)$.
On the other hand, we have $\hat{\varphi}(P') = P$ and $\hat{\varphi}(Q') = [m]Q$
since $\hat{\varphi}\circ \varphi = [m]$.
Therefore, the dual isogeny $\hat{\varphi}$
is an isogeny between enhanced elliptic curves of level $\Gamma(N)$
if and only if $m \equiv 1 \pmod{N}$.
If $m \equiv -1 \pmod{N}$,
then $-\hat{\varphi}$ is an isogeny between enhanced elliptic curves of level $\Gamma(N)$.
However, in the other cases, there is no guarantee that
there exists an isogeny from $(E', (P', Q'))$ to $(E, (P, Q))$ of degree $m$.
Therefore, the modular polynomial for $\Gamma(N)$
is not necessarily symmetric in general.
A similar argument holds for $\Gamma_1(N)$.
Conversely, if $m \equiv \pm 1 \pmod{N}$ or $\Gamma = \Gamma_0(N)$,
then we can show that the modular polynomial is symmetric.
\begin{theorem}\label{thm:symmetry_of_modular_polynomial}
    Let $m > 1$ be an integer prime to $N$.
    Suppose that
    $m \equiv \pm 1 \pmod{N}$ or $\Gamma = \Gamma_0(N)$.
    Then, the modular polynomial $\Phi_m^{\alpha}(X, Y)$ is symmetric,
    i.e., $\Phi_m^{\alpha}(X, Y) = \Phi_m^{\alpha}(Y, X)$.
\end{theorem}
\begin{proof}
    Let $z, z' \in \image \alpha$ such that
    there exists a cyclic isogeny $\varphi$ of degree $m$
    between the enhanced elliptic curves
    $(E_z^\alpha, \ast)$ and $(E_{z'}^\alpha, \ast)$.
    Then, as discussed above,
    the dual isogeny $\hat{\varphi}$ or its negation
    is a cyclic isogeny of degree $m$
    between $(E_{z'}^\alpha, \ast)$ and $(E_z^\alpha, \ast)$.
    Therefore, we have $\Phi_m^{\alpha}(z, z') = 0$
    if and only if $\Phi_m^{\alpha}(z', z) = 0$
    for any $z, z' \in \image \alpha$.

    Let $c(Y) \in \Z[Y]$ be the leading coefficient of $\Phi_m^{\alpha}(X, Y)$ in $X$.
    We define a polynomial $\Upsilon(X, Y)$ as $\Phi_m^\alpha(Y, X) - c(X)\Phi_m^{\alpha}(X, Y)$.
    Then, for any $z, z' \in \image \alpha$ such that $\Phi_m^{\alpha}(z, z') = 0$,
    we have $\Upsilon(z, z') = 0$ by the above argument.
    In addition, we have $\deg_Y(\Upsilon) < \psi(m)$.
    Therefore, we have $\Upsilon(X, Y) = 0$,
    thus $\Phi_m^{\alpha}(Y, X) = c(X) \Phi_m^{\alpha}(X, Y)$.
    By exchanging the roles of $X$ and $Y$,
    we have $\Phi_m^{\alpha}(X, Y) = c(Y) \Phi_m^{\alpha}(Y, X)$.
    This implies $\Phi_m^{\alpha}(X, Y) = c(X) c(Y) \Phi_m^{\alpha}(X, Y)$.
    Therefore, we have $c(X) = c(Y) = \pm 1$.
    
    Suppose that $c(Y) = -1$.
    Then, we have $\Phi_m^{\alpha}(X, X) = -\Phi_m^{\alpha}(X, X)$,
    thus $\Phi_m^{\alpha}(X, X) = 0$.
    This means that,
    for any $z \in \image \alpha$,
    the elliptic curve $E_z^\alpha$ is $m$-isogenous to itself.
    This is a contradiction since $m > 1$.
    Therefore, we obtain $c(Y) = 1$,
    thus $\Phi_m^{\alpha}(X, Y) = \Phi_m^{\alpha}(Y, X)$.
\end{proof}

If $\alpha$ is computable in the sense of Definition \ref{def:computable},
then we can prove the irreducibility of the modular polynomial for $\alpha$.
\begin{theorem}\label{thm:irreducibility_modular_polynomial}
    Assume that the invariant $\alpha$ is computable.
    Let $m$ be a positive integer prime to $N$.
    Then, the modular polynomial $\Phi_m^\alpha(X, Y)$
    is irreducible as a polynomial in $Y$ over $\Comp(X)$.
\end{theorem}
\begin{proof}
    Consider the univariate rational function field $\Comp(j)$ over $\Comp$
    and its algebraic closure $\overline{\Comp(j)}$.
    Let $\tilde{\jmath}$ be a root of $\Phi_m(j, X)$ in $\overline{\Comp(j)}$,
    and $K$ be the minimal splitting field of $J_1(X) - J_0(X)j$ in $\overline{\Comp(j)}$.

    Let $E_j^\mathrm{jinv}$ be the elliptic curve over $\Comp(j)$
    defined in Section \ref{sec:invariants_and_models}.
    Then it is known that
    the Galois group of $\Comp(j, E_j^\mathrm{jinv}[n])$ over $\Comp(j)$
    is isomorphic to $\mathrm{SL}_2(\Z/n\Z)$ for any positive integer $n$
    (see \cite[Section 7.5]{diamond2006first}).
    Since $m$ is prime to $N$,
    we have $\mathrm{SL}_2(\Z/mN\Z) \cong \mathrm{SL}_2(\Z/m\Z) \times \mathrm{SL}_2(\Z/N\Z)$.
    Under this isomorphism,
    the subgroups $\mathrm{SL}_2(\Z/N\Z)$ and $\mathrm{SL}_2(\Z/m\Z)$ of $\mathrm{SL}_2(\Z/mN\Z)$
    correspond to
    the intermediate fields $\Comp(j, E_j^\mathrm{jinv}[m])$ and $\Comp(j, E_j^\mathrm{jinv}[N])$
    of $\Comp(j, E_j^\mathrm{jinv}[mN])/\Comp(j)$, respectively.
    Therefore, we have
    $\Comp(j, E_j^\mathrm{jinv}[m]) \cap \Comp(j, E_j^\mathrm{jinv}[N]) = \Comp(j)$.

    Let $S$ be a level structure of level $\Gamma$ on $E_j^\mathrm{jinv}$.
    Since the invariant $\alpha$ is computable,
    we can define the invariant $\tilde{\alpha} = \alpha(E_j^\mathrm{jinv}, S)$
    in $\overline{\Comp(j)}$,
    and $\tilde{\alpha}$ is contained in $\Comp(j, E_j^\mathrm{jinv}[N])$.
    This implies that $K$ is contained in $\Comp(j, E_j^\mathrm{jinv}[N])$.
    On the other hand,
    we have $\Comp(j, \tilde{\jmath}) \subset \Comp(j, E_j^\mathrm{jinv}[m])$
    since $\tilde{\jmath}$ can be obtained by V\'elu's formula~\cite{Velu1971}.
    Therefore, we have $K \cap \Comp(j, \tilde{\jmath}) = \Comp(j)$.

    Let $\varphi(X, Y)$ be a monic factor of $\Phi_m^\alpha(X, Y)$
    in $\Comp(X)[Y]$.
    Since $\Phi_m^\alpha(X, Y)$ is monic in $Y$ and has coefficients in $\Comp[X]$,
    we have $\varphi(X, Y) \in \Comp[X][Y]$.
    Let $d$ be the degree of $\varphi(X, Y)$ in $Y$.
    As we showed in the proof of Theorem \ref{thm:integral_modular_polynomial},
    we have
    \begin{equation*}
        \res_Z(\res_W(F(W, X), \Phi_m^\alpha(W, Z)), F(Z, Y)) = \gamma\cdot\Phi_m(X, Y)^{\deg J_1}
    \end{equation*}
    for some nonzero constant $\gamma \in \Comp$.
    By the irreducibility of $\Phi_m(X, Y)$ in $\Comp(X)[Y]$
    and the degree consideration,
    we have
    \begin{equation}\label{eq:resultant_decomposition}
        \res_Z(\res_W(F(W, X), \varphi(W, Z)), F(Z, Y)) = \gamma'\cdot\Phi_m(X, Y)^D
    \end{equation}
    for some nonzero constant $\gamma' \in \Comp$,
    where $D = d \cdot \deg J_1 / \psi(m)$.

    Let $\tilde{\alpha}$ be a root of $F(X, j)$ in $\overline{\Comp(j)}$.
    By the definition of modular polynomials,
    we have $\res_Y(\Phi_m^\alpha(\tilde{\alpha}, Y), F(Y, \tilde{\jmath})) = 0$.
    Therefore, we can take the above $\varphi(X, Y)$
    such that $\res_Y(\varphi(\tilde{\alpha}, Y), F(Y, \tilde{\jmath})) = 0$.
    Since $K \cap \Comp(j, \tilde{\jmath}) = \Comp(j)$,
    for any two roots $\tilde{\alpha}_1, \tilde{\alpha}_2$ of $F(X, j)$,
    there exists an automorphism $\sigma$ of $\overline{\Comp(j)}$ over $\Comp(j)$
    such that $\sigma(\tilde{\alpha}_1) = \tilde{\alpha}_2$ and $\sigma(\tilde{\jmath}) = \tilde{\jmath}$.
    This implies that
    $\res_Y(\varphi(\tilde{\alpha}, Y), F(Y, \tilde{\jmath})) = 0$
    for any root $\tilde{\alpha}$ of $F(X, j)$.

    By substituting $X = j$ in the left-hand side of \eqref{eq:resultant_decomposition},
    we have
    \begin{equation*}
        \res_Z(\res_W(F(W, j), \varphi(W, Z)), F(Z, Y)) = \prod_{\tilde{\alpha}} \res_Z(\varphi(\tilde{\alpha}, Z), F(Z, Y)),
    \end{equation*}
    where $\tilde{\alpha}$ runs through all roots of $F(X, j)$.
    Since each factor $\res_Z(\varphi(\tilde{\alpha}, Z), F(Z, Y))$ has a root $\tilde{\jmath}$,
    the equation \eqref{eq:resultant_decomposition} implies that
    $(j - \tilde{\jmath})^{\deg J_1}$ divides $\Phi_m(j, Y)^D$.
    By the irreducibility of $\Phi_m(j, Y)$ in $\Comp(j)[Y]$,
    we have $D = \deg J_1$, thus $d = \psi(m)$.
    Therefore, $\varphi(X, Y) = \Phi_m^\alpha(X, Y)$,
    which shows the irreducibility of $\Phi_m^\alpha(X, Y)$ in $\Comp(X)[Y]$.
\end{proof}

\subsection{Unproved conjectures}\label{subsec:conjectures}
Here, we state two unproved conjectures
related to our modular polynomials.

The first conjecture is about bounds of the coefficients of $\Phi_m^\alpha$.
For this purpose,
we define the \emph{logarithmic height} of a polynomial as follows:
\begin{definition}\label{def:log_height}
    For a polynomial $F(X, Y) = \sum_{i, j} c_{i, j} X^i Y^j \in \Z[X, Y]$,
    we define the \emph{logarithmic height} of $F(X, Y)$ as
    \begin{equation*}
        h(F) = \max_{i, j} \log|c_{i, j}|,
    \end{equation*}
    where $\log$ denotes the natural logarithm,
    and we set $\log 0 = 0$.
\end{definition}
The bound of the logarithmic height of the classical modular polynomial $\Phi_\ell$
for a prime number $\ell$
is given by Theorem 1 and Corollary 9 in \cite{BS2010heigh_modular}:
For every prime number $\ell$, we have
\begin{equation}\label{eq:height_bound_j}
    h(\Phi_\ell) \leq 6\ell\log\ell + 16\ell + \min\{2\ell, 14\sqrt{\ell}\log\ell\}.
\end{equation}
We denote by $B_\ell$ the right-hand side of \eqref{eq:height_bound_j}.
Our conjecture is as follows:
\begin{conjecture}\label{conj:height_bound}
    There exists a bound $L$ depending on $\alpha$
    such that,
    for any prime number $\ell > L$ not dividing $N$, we have
    \begin{equation*}
        h(\Phi_\ell^{\alpha}) \leq B_\ell / \deg(J_1).
    \end{equation*}
\end{conjecture}
This conjecture is a natural generalization of a heuristic
stated in the equation (23) in \cite{BS2010heigh_modular}.
Note that $\deg(J_1) = [\SL : \Gamma]$
as we mentioned in Remark \ref{rem:degree_J1}.
We can justify this conjecture by the following discussion.
As we showed in the proof of Theorem \ref{thm:integral_modular_polynomial},
we have
\begin{equation*}
    \res_Z(\res_W(F(W, X), \Phi_\ell^\alpha(W, Z)), F(Z, Y)) = \gamma \cdot \Phi_\ell(X, Y)^{\deg(J_1)},
\end{equation*}
where $F(X, Y) = J_1(X) - J_0(X) Y$
and $\gamma$ is some integer independent of $\ell$.
Roughly speaking,
the logarithmic height of the left-hand side
is about $\deg(J_1)^2 \cdot h(\Phi_\ell^\alpha)$,
and that of the right-hand side is about $\deg(J_1) \cdot h(\Phi_\ell)$.
This gives a heuristic reason for the conjecture.

The second conjecture is about the evaluation of our modular polynomials
at a point not contained in the image of the invariant.
\begin{conjecture}\label{conj:evaluation_at_non_image}
    For any positive integer $m$ prime to $N$,
    there exists a permutation $\sigma$ on $\Comp\setminus\image \alpha$
    such that
    \begin{equation*}
        \Phi_m^{\alpha}(z, Y) = (Y - \sigma(z))^{\psi(m)},
    \end{equation*}
    for any $z \in \Comp\setminus\image \alpha$.
\end{conjecture}
It is known that
the modular polynomial for the Legendre $\lambda$-invariant
satisfies this property with the identity permutation
(see exercise 6 in Chap. 4.4 in \cite{PiAGM}).
The proof is based on the fact that the cusps $0$ and $i_\infty$ of $X(2)$
are invariant under the action of certain elements of $\operatorname{GL}_2(\Z)$.
Thus, this proof cannot be directly applied to our modular polynomials.
Our experimental results show that
the permutation $\sigma$ in Conjecture \ref{conj:evaluation_at_non_image}
could be nontrivial in general.
See Section \ref{sec:experiment} for details.

\section{Algorithm}\label{sec:algorithms}



In this section, we present an algorithm to compute
the modular polynomial $\Phi_{m}^{\alpha}$ of order $m$ for the invariant $\alpha$.
To do so, we assume that the invariant $\alpha$ is computable.
Then,
we can define the invariant $\alpha$ over a finite field whose characteristic does not divide $N$
by considering the reduction of the rational function in \Cref{def:computable}.
Therefore, we denote by $\alpha(E, *)$
the invariant $\alpha$ of an enhanced elliptic curve $(E, *)$ over a finite field.

In the following, we restrict our attention to the case where $m=\l$ is an odd prime number.
The modular polynomial $\Phi_m^\alpha$ for a composite number $m$ can be computed from
that for prime numbers dividing $m$ as in the classical case (see \cite[\S 69]{Weber1908}).
In the case where $\l=2$,
we can compute $\Phi_2^\alpha$ by using
the polynomial-interpolation method of \cite{CL2005modular},
whose asymptotic complexity in $\l$ is larger than that of
the algorithm for odd primes $\l$ presented in this section.

Our algorithm is based on the algorithm of Kunzweiler--Robert~\cite{deformation_Kunzweiler_2024}
for the classical modular polynomials.
Therefore,
we first briefly review the deformation theory of elliptic curves,
which is required to understand the Kunzweiler--Robert algorithm,
and then explain their algorithm to compute the classical modular polynomial.
Finally, we present our algorithm to compute the modular polynomial $\Phi_{\l}^{\alpha}$.

\subsection{Deformation}\label{subsec:deformation}

We first briefly review the deformation theory of elliptic curves and related algorithms. 
For more details, see \cite{deformation_Kunzweiler_2024}.

Throughout this section, 
let $p\ge 5$ and $\l\ge 3$ be different prime numbers,
and $R$ be an Artinian ring $\Fps[\epsilon]/(\epsilon^{\l+2})$.  

\begin{definition}  
    An \emph{elliptic curve over $R$} is a pair $(\dE,\tilde{0})$ where $\dE$ is a smooth,
    proper, and geometrically connected curve of relative dimension~$1$ over
    $\Spec R$, and 
    $\tilde{0} : \operatorname{Spec}R \to \dE$
    is a section, called the \emph{origin}, such that each geometric fiber of
    $\dE \to \operatorname{Spec}R$ is a smooth projective curve of genus~$1$, and the
    section $\tilde{0}$ endows every fiber with the structure of a commutative group
    scheme.
\end{definition}

The unique maximal ideal of $R$ induces the surjection to the residue field $R\to \Fps$, 
which induces the closed immersion
$\Spec \Fps \to \Spec R$. 
For an elliptic curve $\dE/R$, 
the elliptic curve $\dE_{\Fps}:=\dE\times_R \Spec \Fps$ over $\Fps$ is called a \emph{special fiber}.  

\begin{definition}
    For an elliptic curve $E$ over $\Fps$, 
    a \emph{deformation} of $E$ over $R$ is an elliptic curve $\dE$ over $R$ such that 
    the special fiber $\dE_{\Fps}$ is isomorphic to $E$ over $\Fps$.
\end{definition}

\begin{theorem}[{\cite[Remark 2.4]{deformation_Kunzweiler_2024}}]\label{thm:unique_lift_point}
For a deformation $\dE/R$ of $E/\Fps$ and a positive integer $N$ prime to $p$,
there exists a natural isomorphism $E[N]\cong \dE[N]$. 
Thus, for $P\in E[N]$, there exists a corresponding point $\tilde{P}\in \dE[N]$. 
\end{theorem}

Let $(E, S)$ be an enhanced elliptic curve over $\Fps$,
where $S$ is a level structure of level $\Gamma$ on $E$.
\Cref{thm:unique_lift_point} assures us that we can define the enhanced elliptic curve $(\dE, \tilde{S})$ over $R$.
By using the rational function in \Cref{def:computable},
we can define the invariant $\alpha(\dE, \tilde{S})$ over $R$. 

For an isogeny $f:E_1\to E_2$ over $\Fps$ and  deformations $\dE_1,\dE_2$ of $E_1,E_2$, 
we say an isogeny $\tilde{f}:\dE_1\to \dE_2$ over $R$ is a \emph{lift} of $f$
if the isogeny
$\tilde{f}_{\Fps}:\dE_{1,\Fps}\to \dE_{2,\Fps}$
between the special fibers
is $f$ under some isomorphisms $\dE_{1,\Fps}\cong E_1$ and $\dE_{2,\Fps}\cong E_2$.
The following theorem states the existence and uniqueness of lifts of isogenies.
\begin{theorem}
[{\cite[Proposition 2.8]{deformation_Kunzweiler_2024}}]
    For an isogeny $f:E_1\to E_2$ over $\Fps$ and a deformation $\dE_1$ of $E_1$ over $R$,
    there is a unique deformation $\dE_2$ of $E_2$ over $R$ such that $f$ lifts to an isogeny $\tilde{f}:\dE_1\to \dE_2$ over $R$. 
\end{theorem}

\subsection{Kunzweiler--Robert algorithm}\label{subsec:KRalgorithm}
The Kunzweiler--Robert algorithm uses the CRT method,
i.e., computing $\Phi_\l \bmod p$ for many primes $p$
and lifting them to $\Z[X,Y]$ by the CRT.

A rough sketch of the Kunzweiler--Robert algorithm to compute
the classical modular polynomial $\Phi_{\l} \bmod p$ is as follows.
\begin{enumerate}
    \item Let $E/\Fps$ be an elliptic curve such that $E[\l]\subset E(\Fps)$.
    \item Take a deformation $\dE/R$ of $E$ such that $j(\dE)=j(E)+\epsilon$.
    \item Compute all isogenies $f:E\to E'$ of degree $\l$ over $\Fps$.
    \item For each isogeny $f:E\to E'$ of degree $\l$, compute the deformation $\dE'$ of $E'$ such that $f$ lifts to an isogeny $\tilde{f}:\dE\to \dE'$.
        \label{item:step_compute_deformation} 
    \item Let $\varphi(Y):=\prod_{\dE'}(Y-j(\dE'))\in R[Y]$, where $\dE'$ runs through all deformations obtained in the previous step.
    \item By substituting $\epsilon=X-j(E)$, obtain $\Phi_{\l}(X,Y) \pmod p$.
\end{enumerate}
Since $R=\Fps[\epsilon]/(\epsilon^{\l+2})$,
the polynomial obtained in the last step is
congruent to $\Phi_{\l}(X,Y)$ modulo $(X-j(E))^{\l+2}$,
and hence it is equal to $\Phi_{\l}(X,Y)$ in $\Fp[X,Y]$.

The core of this algorithm is step \ref{item:step_compute_deformation}.
Suppose that $E[\l] \subset E(\Fps)$.
Although we can compute the lift $\tilde{f}$ in step \ref{item:step_compute_deformation}
by V\'elu's formula \cite{Velu1971},
this leads to an inefficient algorithm.
Since the complexity of one multiplication in $R$ is $\tilde{O}(\l)$ operations in $\Fps$,
the total complexity to compute $\Phi_{\l} \bmod p$
becomes $\tilde{O}(\l^{\frac{5}{2}})$ not $\tilde{O}(\l^2)$
even if we use the $\sqrt{\vphantom{l}}$\'elu's formula \cite{BFLS2020}.

Instead, we use \emph{Kani's isogeny diamond}~\cite{Kan1997number}.
Suppose that there exists an endomorphism $\gamma\in \End(E)$
such that $\deg\gamma + \l = 2^n$ and $\l \nmid \deg \gamma$
for some $n \in \mathbb{N}$.
Then we have the following diagram of isogenies over $\Fps$ and their lifts over $R$:
\begin{equation*}\label{eq:diagram_l3}
    \begin{tikzcd}
        E \arrow[r,"f"] \arrow[d,"\gamma"'] &
        E' \arrow[d,"\gamma'"] \\
        E \arrow[r,"f'"'] &
        E'',
    \end{tikzcd}
    \hspace{3em}
    \begin{tikzcd}
        \dE \arrow[r,"\tilde{f}"] \arrow[d,"\tilde{\gamma}"'] &
        \dE_1 \arrow[d,"\tilde{\gamma}'"] \\
        \dE_2 \arrow[r,"\tilde{f}'"'] &
        \dE_3,
    \end{tikzcd}
\end{equation*}
where $\ker\gamma'=f(\ker\gamma)$ and $\ker f' = \gamma(\ker f)$.
It is known that the isogenies in the left diagram can be obtained by computing
a $(2^n, 2^n)$-isogeny $E \times E'' \to E' \times E$
(see Theorem 1 in \cite{EC:MMPPW23} for more details).
Based on this,
\cite{deformation_Kunzweiler_2024} provides an algorithm
to compute the deformation
$\dE_1,\dE_2,\dE_3$
from $E, E', E'', \dE$,
and a generator of the kernel of the $(2^n,2^n)$-isogeny $E \times E'' \to E' \times E$,
which is called $\mathtt{lift\_isogeny\_diamond}$~\cite[Algorithm 4]{deformation_Kunzweiler_2024}.
The complexity of $\mathtt{lift\_isogeny\_diamond}$ is $O(n)$ operations in $R$,
thus the total complexity to compute $\Phi_{\l}\pmod p$ becomes $\tilde{O}(n\l^2)$.
By taking $2^n \approx \l$,
we achieve the complexity $\tilde{O}(\l^2)$.

To realize this algorithm,
we choose the elliptic curve $E$ as a curve over $\Fp$
with complex multiplication by $\Z[2\sqrt{-1}]$
\footnote{
    We cannot use a curve with $j$-invariant $0$ or $1728$.
    The reason is explained in \cite[Example 2.7]{deformation_Kunzweiler_2024}.
},
and $p$ such that there exists integers $n,a,b$ such that $2^n - \l = a^2 + 4b^2$
and $2^n\l \mid (p+1)$ if $\l \equiv 3 \pmod 4$.
Then, $E$ is supersingular and $E(\Fps) = E[p + 1]$ contains $E[\l]$ and $E[2^n]$.
Since $E$ has complex multiplication by $\Z[2\sqrt{-1}]$,
we can take $\gamma \in \End(E)$ with $\deg \gamma = a^2 + 4b^2 = 2^n - \l$.
If $\l \equiv 1 \pmod 4$,
then the equation $2^n - \l = a^2 + b^2$ never holds.
In this case, we use the equation $2^n - 3\l = a^2 + 4b^2$ instead,
and use the diagram shifted by isogenies of degree $3$
(see \cite[Appendix A]{deformation_Kunzweiler_2024} for more details).
This leads to the definition of a set of primes $\Pl$:
\[\Pl:=\{p>11~ \text{prime}\mid \exists n,a,b~\text{with}~2^n-\l=a^2+4b^2~\text{and}~2^n c_{\l}\l|(p+1)\},\]
where $c_{\l}=1$ if $\l\equiv 3\pmod 4$ and $c_{\l}=3$ if $\l\equiv 1\pmod 4$.

Based on the above discussion,
Kunzweiler--Robert show that,
for any prime $p \in \Pl$ such that $\log p \in O(\log \l)$,
the complexity of computing $\Phi_{\l} \bmod p$ is
$O(\l^2 \log^3 \l \log\log \l)$ bit operations
(Theorems 4.4 and 4.5 in \cite{deformation_Kunzweiler_2024}).
Since the logarithmic height of $\Phi_{\l}$ is bounded by
$B_\l \in O(\l\log\l)$ \eqref{eq:height_bound_j},
it suffices to compute $\Phi_{\l} \bmod p$ for $O(\l)$ primes $p$
with $\log p \in O(\log \l)$.
For this, they assume the following heuristic:
\begin{heuristic}[{\cite[Heuristic 4.1]{deformation_Kunzweiler_2024}}]\label{heur:cardinality_n}
    Let $\l$ be a positive integer and $c_{\l}$ be as above.
    Then, we have
    \begin{equation*}
        \#\{n \in \mathbb{N} \mid 2^n - c_{\l}\l = a^2 + 4b^2,\ n \leq x\} \approx x/\sqrt{x}.
    \end{equation*}
\end{heuristic}
Under this heuristic,
there are $\Omega(\l)$ primes $p \in \Pl$ with $\log p \in O(\log \l)$
(Lemma 4.3 in \cite{deformation_Kunzweiler_2024}),
and we obtain that
the total complexity to compute the classical modular polynomial $\Phi_{\l}$ is
$O(\l^3 \log^3 \l \log\log \l)$ bit operations.

Moreover, 
we can eliminate \Cref{heur:cardinality_n}
by using the trick of a dimension 8 embedding as explained in \cite[\S 3.5]{deformation_Kunzweiler_2024}.
Then, we use 
$
    \Psl:=\{p>11~ \text{prime}\mid 2^n\l|(p+1)\}
$
where $n=\lceil\log_2(\l)\rceil$ 
instead of $\Pl$. 
See Theorem 5.2 in \cite{deformation_Kunzweiler_2024} for more details.

\subsection{Algorithms to compute modular polynomials for invariants}
\label{subsec:algorithms_enhanced}
We present an algorithm to compute the modular polynomials $\Phi_{\l}^{\alpha}(X,Y)\in\Z[X,Y]$.

First, we explain how to compute $\Phi_{\l}^{\alpha} \bmod p$.
For simplicity,
we consider the case where $\Gamma=\Gamma_1(N)$.
The other cases are obtained similarly.
Our algorithm to compute $\Phi_{\l}^{\alpha} \bmod p$ is as follows:
\begin{enumerate}
    \item Let $(E, P)$ be an enhanced elliptic curve over $\Fps$.
    \item Take a lift $\tilde{a}$ such that $\tilde{a} = \alpha(E, P) + \epsilon$,
        and let $\dE$ be a deformation of $E$ such that $j(\dE) = j_{ \alpha }(\tilde{a})$.
    \item Compute all isogenies $f:E\to E'$ of degree $\l$ and the images $P'=f(P)$.
    \item For each isogeny $f:E\to E'$ of degree $\l$,
        compute the deformation $\dE'$ of $E'$ such that $f$ lifts to an isogeny $\tilde{f}:\dE\to \dE'$.
    \item Compute the lift $\tilde{a}'$ of $\alpha(E', P')$ such that $j_{\alpha}(\tilde{a}') = j(\dE')$.
        \label{item:step_compute_lifted_invariant}
    \item Let $\varphi(Y):=\prod_{\tilde{a}'}(Y - \tilde{a}')\in R[Y]$, where $\tilde{a}'$
        runs through all lifts obtained in the previous step.
    \item By substituting $\epsilon= X - \alpha(E,P)$, obtain $\Phi_{\l}^\alpha(X,Y) \pmod p$.
\end{enumerate}
The main difference from the Kunzweiler--Robert algorithm
is step (\ref{item:step_compute_lifted_invariant}).
Thus, we explain how to compute the lift $\tilde{a}'$ in this step.

Let $F(X) = J_1(X) - j(\dE') J_0(X) \in R[X]$.
By the definition of $\alpha$,
the condition $j_{\alpha}(\tilde{a}') = j(\dE')$
is equivalent to $F(\tilde{a}') = 0$.
Therefore, if $\frac{d F}{d X}(\alpha(E', P')) \not\equiv 0 \pmod{(\epsilon)}$,
we can compute $\tilde{a}'$ from $\alpha(E', P')$ and $j(\dE')$
by using $\mathtt{Newton\_lift}$~\cite[Algorithm 1]{deformation_Kunzweiler_2024}. 
The equation $\frac{d F}{d X}(\alpha(E', P')) \equiv 0 \pmod{(\epsilon)}$ holds
if and only if $J_1(X) - j(E') J_0(X)$ has a multiple root at $X = \alpha(E', P')$.
This happens only when $j(E')$ is $1728$ or $0$.
If this case occurs, we can simply discard $p$ and choose another prime.

Note that we have $\tilde{a}' = \alpha(\dE', \tilde{P}')$,
where $\tilde{P}'$ is the unique lift of $P'$ in $\dE'$ by \Cref{thm:unique_lift_point},
since $\alpha(\dE', \tilde{P}')$ is the unique lift of $\alpha(E', P')$
satisfying $F(\alpha(\dE', \tilde{P}')) = 0$ by Hensel's lemma.
Therefore, at the final step,
we obtain $\Phi_{\l}^\alpha(X,Y) \bmod{(X - \alpha(E,P))^{\l+2}}$
as in the Kunzweiler--Robert algorithm,
which is equal to $\Phi_{\l}^\alpha(X,Y) \bmod p$.

Based on the above algorithm to compute $\Phi_{\l}^\alpha \bmod p$,
we can compute $\Phi_{\l}^\alpha(X,Y) \in \Z[X,Y]$.
The following theorem states the complexity of our algorithm.
\begin{theorem}\label{thm:complexity}
    Let $\l$ be an odd prime number prime to $N$.
    Assume that $N \in O(1)$.
    Then,
    under \Cref{conj:height_bound},
    we can compute $\Phi_\ell^\alpha$ in
    $O(\l^3 \log^3 \l \log\log \l)$ bit operations.
\end{theorem}
\begin{proof}
    The proof is similar to that of \cite[Theorem 4.4]{deformation_Kunzweiler_2024}.
    Here, we only explain the differences.
    Since we require the level structure to be defined over $\Fps$,
    we use a set of primes $\PlN$ defined as
    \[
        \PlN:=\{p>11~ \text{prime}\mid \exists n,a,b~\text{with}~2^n-c_{\l}\l=a^2+4b^2~\text{and}~2^nc_{\l}\l|(p+1)~\text{and}~N\mid(p+1)\}
    \]
    for computing $\Phi_{\l}^\alpha \bmod p$.
    By a similar argument as in Lemma 4.3 in \cite{deformation_Kunzweiler_2024},
    we can show that
    $\PlN$ has $\Omega(\l)$ primes $p$ with $\log p \in O(\log \l)$.
    By \Cref{conj:height_bound},
    $O(\l)$ primes are sufficient to recover $\Phi_{\l}^\alpha$ by the CRT method.

    Since $\mathtt{Newton\_lift}$ is also used in the Kunzweiler--Robert algorithm
    for each isogeny $f$,
    the computation of the lift $\tilde{a}'$ in step (\ref{item:step_compute_lifted_invariant})
    does not increase the overall complexity.
    Therefore, the total complexity of our algorithm is
    $O(\l^3 \log^3 \l \log\log \l)$ bit operations, the same 
    as that of the Kunzweiler--Robert algorithm.
\end{proof}
As in the Kunzweiler--Robert algorithm,
we can eliminate \Cref{heur:cardinality_n}
by using the trick of a dimension 8 embedding
and use a set of primes defined as
$\{p>11~ \text{prime}\mid 2^n\l N|(p+1)\}$
where $n=\lceil\log_2(\l)\rceil$ instead of $\PlN$.

\section{Examples}\label{sec:examples}
We give two examples of invariants with good models:
the coefficient $A$ of Montgomery curves
and the coefficient $d$ of Hessian curves.
Furthermore, we also show that the modular polynomials for these invariants
have zero-coefficient terms under certain degree conditions.

\subsection{Montgomery curve}\label{subsec:montgomery}
A \emph{Montgomery curve} over $\Comp$ is an elliptic curve defined by
the equation
\begin{equation*}
    E_A^\mont : Y^2Z = X^3 + A X^2 Z + X Z^2 \text{ for } A \in \Comp,\ A^2 \neq 4
\end{equation*}
and a point $(0 : 1 : 0)$.
For this curve,
the set of rational points
\begin{equation*}
    C_A^\mont = \{(0:1:0), (0:0:1), (1 : \pm\sqrt{A + 2}: 1)\}
\end{equation*}
forms a cyclic subgroup of order $4$.
Let $S_0(4)$ be the set of isomorphism classes of
enhanced elliptic curves of level $\Gamma_0(4)$.
We define a map
\begin{equation*}
    \mont : S_0(4) \to \Comp \setminus \{\pm 2\};\
    \EEparamZero{A}{\mont} \mapsto A.
\end{equation*}
Then, we have the following proposition.
\begin{proposition}\label{prop:montgomery_well_defined}
    The map $\mont$ is well-defined
    and gives
    a computable invariant of level $\Gamma_0(4)$.
\end{proposition}
\begin{proof}
    First, we show that $\mont$ is well-defined.
    Let $(E^\mathrm{jinv}_j, C_j)$ be an enhanced elliptic curve of level $\Gamma_0(4)$
    with $j$-invariant $j$.
    By shifting the $x$-coordinate,
    $(E^\mathrm{jinv}_j, C_j)$ is isomorphic to
    an enhanced elliptic curve $(E, C)$
    where
    \begin{equation*}
        E : Y^2 Z = X^3 + a_2 X^2 Z + a_4 XZ^2,
    \end{equation*}
    and the point of order $2$ in $C$ is $(0 : 0 : 1)$.
    Let $(x_1 : y_1 : 1)$ be a generator of $C$.
    An easy calculation shows that $x_1^2 = a_4$.
    Therefore, the isomorphism
    \begin{equation*}
        E \to E_{a_2/x_1}^\mont;\ (X : Y : Z) \mapsto (X/x_1 : Y/\sqrt{x_1}^3 : Z)
    \end{equation*}
    sends $C$ to $C_{a_2/x_1}^\mont$.
    This means that $(E, C)$ is isomorphic to $(E_{a_2/x_1}^\mont, C_{a_2/x_1}^\mont)$.
    
    From Proposition 1 in \cite{PKC:OnuMor22},
    it follows that
    $(E_A^\mont, C_A^\mont)$ is isomorphic to $(E_{A'}^\mont, C_{A'}^\mont)$
    if and only if $A = A'$.
    This means that $\mont$ is well-defined.

    Since $a_2$ and $x_1$ can be represented
    as rational functions in $j$ and the coordinates of a generator of $C_j$,
    the invariant $\mont$ is computable.

    The $j$-invariant of $E_A^\mont$ is given by
    \begin{equation*}
        j(E_A^\mont) = \frac{2^8(A^2 - 3)^3}{A^2 - 4}.
    \end{equation*}
    It is easy to see that this satisfies the required conditions.
\end{proof}
Next, we show that the affine $x$-coordinate on Montgomery curves
gives a good model of $\mont$.
\begin{proposition}
    For $A \in \Comp \setminus \{\pm 2\}$,
    we define the function
    \begin{equation*}
        x_A : E_A^\mont(\Comp) \to \Comp \cup \{\infty\};\
        (X : Y : Z) \mapsto
        \begin{cases}
            X/Z & \text{if } Z \neq 0,\\
            \infty & \text{if } Z = 0.
        \end{cases}
    \end{equation*}
    Then, $\{(\EEparamZero{A}{\mont}, x_A)\}$
    is a good model of $\mont$.
\end{proposition}
\begin{proof}
    It is easy to see that
    $\{(\EEparamZero{A}{\mont}, x_A)\}$
    is a model of $\mont$.

    Proposition 21 in \cite{MOAT2022gen_montgomery}
    gives scalar multiplication formulas
    and division polynomials for this model.
    In addition,
    Theorem 1 in \cite{AC:CosHis17} gives isogeny formulas for this model.
\end{proof}

From the theorems in Section \ref{subsec:main_theorem},
we obtain the following corollary.
\begin{corollary}
    For any positive integer $m$ prime to $2$,
    there exists the modular polynomial $\Phi_m^{\mont}(X, Y) \in \Z[X, Y]$
    of order $m$ for $\mont$.
    In addition, if $m > 1$,
    then $\Phi_m^{\mont}(X, Y)$ is symmetric.
\end{corollary}

As stated in \S 7 in \cite{broker2012modular},
the coefficient of $X^iY^j$ in the modular polynomial of order $m$ for a modular function of level $N$
is zero if $i, j$, and $m$ satisfy some congruence conditions modulo $N$.
We can show similar properties for our modular polynomials.
\begin{proposition}\label{prop:coeff_zero_montgomery}
    Let $m > 1$ be an odd integer.
    The coefficient of $X^i Y^j$ in $\Phi_m^{\mont}(X, Y)$
    is zero unless $i + j \equiv 0 \pmod{2}$.
\end{proposition}
\begin{proof}
    Let $A$ be in $\image\mont$,
    and $A'$ be a root of $\Phi_m^{\mont}(A, Y)$.
    Then, there exists a cyclic isogeny $\varphi$ of degree $m$
    between the enhanced elliptic curves
    $(E_A^\mont, C_A^\mont)$ and $(E_{A'}^\mont, C_{A'}^\mont)$.
    
    Let $\iota_A$ be the isomorphism $E_A^\mont \to E_{-A}^\mont$
    defined by $(X : Y : Z) \mapsto (-X : \sqrt{-1} Y : Z)$,
    and $\iota_{A'}$ be the similar isomorphism $E_{A'}^\mont \to E_{-A'}^\mont$.
    Then, it is easy to see that
    the composition $\iota_{A'} \circ \varphi \circ \iota_A^{-1}$
    is a cyclic isogeny of degree $m$
    between $(E_{-A}^\mont, C_{-A}^\mont)$ and $(E_{-A'}^\mont, C_{-A'}^\mont)$.
    Therefore, we have $\Phi_m^{\mont}(-A, -A') = 0$.
    This means that the polynomial $\Phi_m^{\mont}(-X, -Y)$
    satisfies the conditions for the modular polynomial of order $m$ for $\mont$.
    By the uniqueness of the modular polynomial,
    we have $\Phi_m^{\mont}(-X, -Y) = \Phi_m^{\mont}(X, Y)$.
    This implies the desired result.
\end{proof}

\subsection{Hessian curve}\label{subsec:hessian}
A \emph{Hessian curve} over $\Comp$ is an elliptic curve defined by
the equation
\begin{equation*}
    E_d^\hess : X^3 + Y^3 + Z^3 = d XYZ \text{ for } d \in \Comp,\ d^3 \neq 27
\end{equation*}
and a point $(1 : -1 : 0)$.
The arithmetic on Hessian curves
can be found in \cite[\S 2.1]{PKC:FarJoy10}.
In particular,
a point $(x : y : z) \in E_d^\hess(\Comp)$ is in the $3$-torsion subgroup
if and only if $xyz = 0$.
Let $\omega = e^{2\pi i / 3}$.
Then,
we define
$P_d^\hess = (-\omega : 1 : 0)$
and
$Q_d^\hess = (0 : -1 : 1)$.
As we show in the following lemma,
the points $P_d^\hess$ and $Q_d^\hess$
form a basis of the $3$-torsion subgroup of $E_d^\hess$,
and the Weil pairing satisfies $e_3(P_d^\hess, Q_d^\hess) = \omega$.

Let $S(3)$ be the set of isomorphism classes of
enhanced elliptic curves of level $\Gamma(3)$.
We define a map
\begin{equation*}
    \hess : S(3) \to \Comp;\
    \EEparam{d}{\hess} \mapsto d.
\end{equation*}
In the following,
we show that this map is well-defined
and gives an invariant of level $\Gamma(3)$.

We first note that,
we can take a representative $(E, (P, Q))$
of any element in $S(3)$
such that
$E$ is given by
\begin{equation}\label{eq:S3_weier}
    E : Y^2 Z + a_1 X Y Z + a_3 Y Z^2 = X^3
\end{equation}
for some $a_1, a_3 \in \Comp$,
and that $P = (0 : 0 : 1)$
(see \cite[Chap. 4 \S 2]{Husemoller2004elliptic}).
In addition,
for $j \in \Comp$ and the level structure $(P_j, Q_j)$ on $E_j^\mathrm{jinv}$,
the coefficients $a_1$ and $a_3$
satisfying $(E_j^\mathrm{jinv}, (P_j, Q_j)) \cong (E, (P, Q))$
can be represented as rational functions in $j$ and the coordinates of $P_j$.

The following lemma gives an isomorphism
from such an elliptic curve $E$ to a Hessian curve.
\begin{lemma}\label{lem:isom_E_to_Hess}
    Let $E$ be the elliptic curve defined by \eqref{eq:S3_weier},
    $P = (0 : 0 : 1)$,
    and $Q = (\xi:\eta:1)$ be an order-$3$ point on $E$
    such that the Weil pairing $e_3(P, Q) = \omega$.
    Then the map $\mathbb{P}^2_\Comp \to \mathbb{P}^2_\Comp$ defined by
    \begin{equation*}
        (X:Y:Z) \mapsto
            \begin{pmatrix}
                \omega a_1 & (2\omega + 1) & (\omega - 1)a_3 \\
                -(\omega + 1)a_1 & -(2\omega + 1) & (-\omega - 2)a_3 \\
                a_1 + 3\frac{a_3}{\xi} & 0 & 0
            \end{pmatrix}
            \begin{pmatrix}
                X \\ Y \\ Z
            \end{pmatrix}
    \end{equation*}
    induces an isomorphism from $E$ to $E_d^\mathrm{Hess}$ for
    $d = \frac{3a_1\xi}{a_1\xi + 3a_3}$.
    In addition, this isomorphism maps $P$ to $(-\omega:1:0)$
    and $Q$ to $(0:-1:1)$.
\end{lemma}
\begin{proof}
    We first note that $\xi \neq 0$
    since the only points with $X$-coordinate $0$ on $E$
    are $P$, $(0:-a_3:1) = -P$, and the identity point $(0:1:0)$.

    For $c', d' \in \Comp$ with $c' \neq 0, d'^3 \neq 27c'$,
    we define a \emph{generalized Hessian curve} $H_{c', d'}$ by
    \begin{equation*}
        H_{c', d'} : X^3 + Y^3 + c'Z^3 = d' XYZ \text{ and } (1 : -1 : 0).
    \end{equation*}
    Let $\gamma$ be a cube root of $c'$.
    Then there exists an isomorphism from $H_{c', d'}$ to $E_{d'/\gamma}^\mathrm{Hess}$
    defined by $(X:Y:Z) \mapsto (X:Y:\gamma Z)$.

    From the proof of Theorem 1 in \cite{PKC:FarJoy10},
    there exists an isomorphism from $E$ to $H_{c', d'}$
    for $c' = a_1^3 - 27 a_3$ and $d' = 3a_1$
    defined by
    \begin{equation}\label{eq:isom_E_to_Hess}
        (X:Y:Z) \mapsto
            \begin{pmatrix}
                \omega a_1 & (2\omega + 1) & (\omega - 1)a_3 \\
                -(\omega + 1)a_1 & -(2\omega + 1) & (-\omega - 2)a_3 \\
                1 & 0 & 0
            \end{pmatrix}
            \begin{pmatrix}
                X \\ Y \\ Z
            \end{pmatrix}.
    \end{equation}
    Therefore, it suffices to show that $a_1 + 3a_3/\xi$ is a cube root of $a_1^3 - 27 a_3$
    for proving the first part of this lemma.

    The $3$rd division polynomial of $E$ is given by
    \begin{equation*}
        3x^4 + a_1^2 x^3 + 3a_1 a_3 x^2 + 3a_3^2x.
    \end{equation*}
    Since $\xi \neq 0$,
    we have $3\xi^3 + a_1^2 \xi^2 + 3a_1 a_3 \xi + 3a_3^2 = 0$.
    From this, it is easy to verify that
    $(a_1 + 3a_3/\xi)^3 = a_1^3 - 27 a_3$.

    It is straightforward to check that
    the above isomorphism maps $P$ to $(-\omega:1:0)$.
    To see that it maps $Q$ to $(0:-1:1)$,
    we calculate the Weil pairing $e_3(P, Q)$.
    For a point $R$ on $E$ of order $3$,
    the Miller function $f_{3, R}(x, y)$
    is the tangent line at $R$ on the curve $E$ whose coefficient of $y$ is $1$
    (see \cite[\S 4.1]{JC:Miller04} for the definition of Miller function).
    From Proposition 8 in \cite{JC:Miller04},
    we have
    $e_3(P, Q) = -\frac{f_{3, P}(Q)}{f_{3, Q}(P)}$.
    An easy calculation shows that $e_3(P, Q) = \omega$ is equivalent to
    \begin{equation*}
        \omega a_1 \xi + (2\omega + 1)\eta + (\omega - 1)a_3 = 0.
    \end{equation*}
    This implies that
    the image of $Q$ under the above isomorphism is $(0:-1:1)$.
\end{proof}

Next, we show that the coefficient $d$ of a Hessian curve
is unchanged under isomorphisms of enhanced elliptic curves of level $\Gamma(3)$.
\begin{lemma}\label{lem:isom_Hess_to_Hess_w}
    Let $d, d' \in \Comp$ with $d^3 \neq 27$ and $d'^3 \neq 27$,
    and $\iota : E_d^\mathrm{Hess} \to E_{d'}^\mathrm{Hess}$ be an isomorphism
    mapping $(-\omega:1:0)$ to $(-\omega:1:0)$.
    Then,
    $d' = \omega^i d$ and $\iota$ is given by
    $(X:Y:Z) \mapsto (X:Y: \omega^{-i} Z)$
    for some $i \in \{0, 1, 2\}$.
\end{lemma}
\begin{proof}
    Let $a_1 = d/3$, $a_3 = (d^3 - 3^3)/3^6$, and $E_{a_1, a_3}$
    be the elliptic curve defined by
    \begin{equation*}
        E_{a_1, a_3} : Y^2 Z + a_1 XYZ + a_3 Y Z^2 = X^3.
    \end{equation*}
    Then, the isomorphism defined in \eqref{eq:isom_E_to_Hess}
    gives an isomorphism from $E_{a_1, a_3}$ to $E_d^\mathrm{Hess}$
    mapping $(0:0:1)$ to $(-\omega:1:0)$.
    We denote this isomorphism by $\iota_d$.
    Similarly, we define $a_1' = d'/3$, $a_3' = (d'^3 - 3^3)/3^6$,
    and the isomorphism $\iota_{d'} : E_{a_1', a_3'} \to E_{d'}^\mathrm{Hess}$.

    Then, the isomorphism
    $\iota_{d'}^{-1} \circ \iota \circ \iota_d : E_{a_1, a_3} \to E_{a_1', a_3'}$
    maps $(0:0:1)$ to $(0:0:1)$.
    From the well-known fact about isomorphisms of elliptic curves in Weierstrass form
    (see Proposition III.3.1 in \cite{silverman2009arithmetic}),
    there exists $u \in \Comp^\times$ such that
    $a_1' = u a_1$, $a_3' = u^3 a_3$,
    and the isomorphism is given by $(X:Y:Z) \mapsto (u^2 X : u^3 Y : Z)$.
    In particular, we have $a_1^3/a_3 = a_1'^3/a_3'$,
    thus $d'^3 = d^3$.
    Therefore, we have $d' = \omega^i d$ for some $i \in \{0, 1, 2\}$.
    From this, we obtain $u = \omega^i$.
    This means that
    $\iota_{d'}^{-1} \circ \iota \circ \iota_d$ is given by
    \begin{equation*}
        (X:Y:Z) \mapsto (\omega^{2i} X : Y : Z).
    \end{equation*}
    An easy calculation using the definitions of $\iota_d$ and $\iota_{d'}$
    shows that
    $\iota$ is given by the desired form.
\end{proof}

Using these lemmas,
we obtain the following proposition.
\begin{proposition}
    The map $\hess$ defined above
    is well-defined
    and gives a computable invariant of level $\Gamma(3)$.
\end{proposition}
\begin{proof}
    The well-definedness of $\hess$ follows from
    Lemmas \ref{lem:isom_E_to_Hess} and \ref{lem:isom_Hess_to_Hess_w}.
    In addition, $\hess$ is computable
    from the expressions of $d$ in Lemma \ref{lem:isom_E_to_Hess}.

    The $j$-invariant of $E_d^\hess$ is given by
    \begin{equation*}
        j(E_d^\hess) = \frac{d^3(d^3 + 6^3)^3}{(d^3 - 3^3)^3}
    \end{equation*}
    (see \cite[\S 2]{PKC:FarJoy10}).
    It is easy to see that this satisfies the required conditions.
\end{proof}
To see the invariant $\hess$ has a good model,
we use the $t$-coordinate arithmetic on Hessian curves,
which is introduced in \cite{FouazouLontouo2023}.
In particular,
we define the $t$-coordinate as follows:
\begin{equation*}
    t_d : E_d^\hess(\Comp) \to \Comp \cup \{\infty\};\
    (X : Y : Z) \mapsto
    \begin{cases}
        XY/Z^2 & \text{if } Z \neq 0,\\
        \infty & \text{if } Z = 0.
    \end{cases}
\end{equation*}
In the following,
we show that $\{(\EEparam{d}{\hess}, t_d)\}$
is a good model of $\hess$.
To do this,
we give scalar multiplication formulas,
division polynomials, and isogeny formulas
for the $t$-coordinate on Hessian curves.

We use the following theorem on scalar multiplication formulas
for the $t$-coordinate on Hessian curves,
which is given in \cite{FouazouLontouo2023}.
\begin{theorem}\label{thm:division_polynomials_hess}
    We define $h(t) = -4t^3 + d^2 t^2 - 2 d t + 1$,
    $P_1(t) = P_2(t) = 1$,
    $V_1(t) = t$, and $V_2(t) = t^4 - d t^2 + 2t$ in $\Z[d][t]$.
    For a positive integer $k$,
    we define polynomials inductively by
    \begin{align*}
        P_{2k+1}(t) &= \frac{-th(t) P_{2k}(t)^2 - V_{2k}(t)^2}{P_{2k-1}(t)},\\
        V_{2k+1}(t) &= \frac{-th(t)^2P_{2k}(t)^4 + t^2V_{2k}(t)^2
            + (-1 + dt)h(t)P_{2k}(t)^2 V_{2k}(t)}
            {V_{2k-1}(t)},\\
        P_{2k+2}(t) &= \frac{tP_{2k+1}(t)^2 - V_{2k+1}}{-h(t)P_{2k}(t)},\\
        V_{2k+2}(t) &= \frac{t P_{2k+1}(t)^4 + t^2 V_{2k+1}(t)^2
            - (-1 + dt) P_{2k+1}(t)^2 V_{2k+1}(t)}
            {V_{2k}(t)}.
    \end{align*}
    Then, the following holds:
    \begin{enumerate}
        \item $P_m(t), V_m(t) \in \Q[d][t]$
            for any positive integer $m$, \label{item:division_polynomials_hess_th1start}
        \item for any $P \in E_d^\hess(\Comp)$ such that $t_d(P) \neq \infty$,
            we have
            \begin{equation*}
                t_d([m]P) = \begin{cases}
                        \frac{V_m(t_d(P))}{P_m(t_d(P))^2}
                        & \text{if } m \text{ is odd},\\
                        -\frac{V_m(t_d(P))}{h(t_d(P))P_m(t_d(P))^2}
                        & \text{if } m \text{ is even},
                    \end{cases}
            \end{equation*} \label{item:division_polynomials_hess_mult}
        \item if $m$ is odd then $V_m(t)$ and $P_m(t)$ do not have common roots, \label{item:division_polynomials_hess_th1mid}
        \item if $m$ is even then $V_m(t)$ and $h(t)P_m(t)$ do not have common roots, \label{item:division_polynomials_hess_th1end}
        \item $V_m(t)$ is monic for any positive integer $m$, \label{item:division_polynomials_hess_prop1start}
        \item the leading coefficient of $P_m(t)$ in $t$ is $m$ if $m$ is odd,
            and $m/2$ if $m$ is even. \label{item:division_polynomials_hess_prop1end}
    \end{enumerate}
\end{theorem}
\begin{proof}
    The properties (\ref{item:division_polynomials_hess_th1start})--(\ref{item:division_polynomials_hess_th1end})
    follow from Theorem 1 in \cite{FouazouLontouo2023}.
    The properties (\ref{item:division_polynomials_hess_prop1start})--(\ref{item:division_polynomials_hess_prop1end})
    follow from Proposition 1 in \cite{FouazouLontouo2023}.
\end{proof}
This theorem gives us the scalar multiplication formulas
for the $t$-coordinate on Hessian curves.
For division polynomials,
we define the $m$-th division polynomial for the $t$-coordinate on Hessian curves
as $P_m(t)$ if $m$ is odd,
and $-h(t) P_m(t)$ if $m$ is even.
Theorem \ref{thm:division_polynomials_hess}
proves the properties required for division polynomials
except for that they have coefficients in $\Z[d]$.
To show this,
we use the following lemma.
\begin{lemma}\label{lem:division_polynomials_hess_recursion}
    For any positive integer $k$,
    we have
    \begin{align*}
        tV_{2k+1} &= \begin{cases}
                    V_k^2V_{k+1}^2 - hP_k^2 P_{k+1}^2
                        (-hP_k^2V_{k+1} + P_{k+1}^2 V_k - dV_kV_{k+1})
                    \text{ if $k$ is even,}\\
                    V_k^2 V_{k+1}^2 - hP_k^2 P_{k+1}^2
                        (P_k^2 V_{k+1} - hP_{k+1}^2 V_k - dV_k V_{k+1})
                    \text{ if $k$ is odd,}
                \end{cases}\\
        P_{2k+1}^2 &= \begin{cases}
                    (hP_k^2V_{k+1} + P_{k+1}^2 V_k)^2
                    \text{ if $k$ is even,}\\
                    (P_k^2 V_{k+1} + hP_{k+1}^2 V_k)^2
                    \text{ if $k$ is odd,}
                \end{cases}\\
        V_{2k} &= \begin{cases}
                    V_k (V_k^3 - d h^2 P_k^4 V_k - 2 h^3 P_k^6)
                    \text{ if $k$ is even,}\\
                    V_k (V_k^3 - d P_k^4 V_k + 2 P_k^6)
                    \text{ if $k$ is odd,}
                \end{cases}\\
        -h P_{2k}^2 &= \begin{cases}
                    -4 h V_k^3 - d^2 h^2 P_k^4 V_k - 2 d h^3 V_k^2 P_k^6 - h^4 P_k^8
                    \text{ if $k$ is even,}\\
                    4 V_k^3 + d^2 P_k^4 V_k + 2 d V_k^2 P_k^6 - P_k^8
                    \text{ if $k$ is odd.}
                \end{cases}
    \end{align*}
    Here, we omit the variable $t$ from the polynomials for simplicity.
\end{lemma}
\begin{proof}
    We can derive differential addition and doubling formulas
    for the $t$-coordinate from the addition formulas in \cite[\S 2.1]{PKC:FarJoy10}.
    From the differential addition formula and
    (\ref{item:division_polynomials_hess_mult}) in Theorem \ref{thm:division_polynomials_hess},
    if $k$ is even, we have
    \begin{equation*}
        \frac{tV_{2k+1}}{P_{2k+1}^2}
        =
        \frac{V_k^2V_{k+1}^2 - hP_k^2 P_{k+1}^2
                        (-hP_k^2V_{k+1} + P_{k+1}^2 V_k - dV_kV_{k+1})}
            {(hP_k^2 V_{k+1} +hP_{k+1}^2 V_k)^2}.
    \end{equation*}
    By taking $P = (0:-1:1)$ in (\ref{item:division_polynomials_hess_mult}) in Theorem \ref{thm:division_polynomials_hess},
    it follows that $V_k(0) = 0$ if and only if $k \not\equiv 0 \pmod{3}$,
    and $P_k(0) = 0$ if and only if $k \equiv 0 \pmod{3}$.
    Therefore, the numerator of the right-hand side is divisible by $t$.
    Since $V_{2k+1}$ and $P_{2k+1}$ are coprime from
    (\ref{item:division_polynomials_hess_th1mid}) and (\ref{item:division_polynomials_hess_th1end}) in Theorem \ref{thm:division_polynomials_hess},
    we obtain the first two equalities for even $k$
    by comparing the degrees and the leading coefficients.

    The other equalities can be shown similarly by using
    the doubling formula for the $t$-coordinate.
\end{proof}
Using this lemma,
we can show that $P_m(t) \in \Z[d][t]$ for any positive integer $m$
by induction on $m$.

Next,
we give isogeny formulas for the $t$-coordinate on Hessian curves.
Theorem 4 in \cite{BDFM2021} gives a formula of isogenies of degree
prime to $3$ on Hessian curves,
which uses $(X:Y:Z)$-coordinates.
The following lemma is a restatement of this theorem for
our $t$-coordinate.
\begin{lemma}\label{lem:isogeny_formula_Hess}
    Let $d \in \image \hess$ and $m$ be a positive integer
    prime to $3$.
    Let $K$ be a point on $E_d^\mathrm{Hess}$ of order $m$,
    and we define $d' \in \Comp$ by
    \begin{equation*}
        d' = c_1 \left(
                \left(2\left\lceil\frac{m}{2}\right\rceil - 3\right)d
                - 6\sum_{i=1}^{\left\lceil\frac{m}{2}\right\rceil - 1}
                    \frac{1}{t_d([i]K)}
                    + 6c_2
            \right)
            \left(
                \prod_{i=1}^{\left\lceil\frac{m}{2}\right\rceil - 1}
                t_d([i]K)
            \right),
    \end{equation*}
    where
    \begin{align*}
    &c_1 = 1 \text{ and } c_2 = 0 \text{ if } m \text{ is odd},\\
    &c_1 = c_2 = -2t_d\left(\left[\frac{m}{2}\right]K\right)^2 + \frac{d^2}{2} t_d\left(\left[\frac{m}{2}\right]K\right) - \frac{d}{2}
    \text{ if } m \text{ is even}.
    \end{align*}
    Then, the map defined by
    \begin{equation*}
        R \mapsto
        \left(
            \prod_{i=0}^{m-1}X(R + [i]K) :
            \prod_{i=0}^{m-1}Y(R + [i]K) :
            \prod_{i=0}^{m-1}Z(R + [i]K)
        \right)
    \end{equation*}
    is an isogeny from $E_d^\mathrm{Hess}$ to $E_{d'}^\mathrm{Hess}$
    with kernel generated by $K$.
\end{lemma}
\begin{proof}
    Note that the specified point on a Hessian curve in \cite{BDFM2021}
    is $(0:-1:1)$, different from our definition.
    Therefore, we need to compose the isomorphism
    defined by $(X:Y:Z) \mapsto (Z:Y:X)$
    to their isogeny formula.

    Let $(X_i, Y_i, Z_i) \in \Comp^3$ be a representative of $[i]K$ for $i = 1, \dots, m - 1$.
    Then,
    Theorem 4 in \cite{BDFM2021} shows that the lemma holds for
    \begin{equation*}
        d' = \left(
                (3 - 2m)d + 6\sum_{i=1}^{m - 1}
                    \frac{X_i^2}{Y_i Z_i}
            \right)
            \left(
                \prod_{i=1}^{m - 1}
                \frac{X_i}{Z_i}
            \right).
    \end{equation*}
    Therefore, it suffices to show that
    this is equal to the expression in the lemma.

    Since $-(X:Y:Z) = (Y:X:Z)$ on Hessian curves,
    we have
    \begin{equation*}
        \frac{X_i^2}{Y_i Z_i} + \frac{X_{m - i}^2}{Y_{m - i} Z_{m - i}}
        = \frac{X_i^3 + Y_i^3}{X_i Y_i Z_i}
        = d - \frac{1}{t_d([i]K)}.
    \end{equation*}
    If $m$ is even, we have
    \begin{equation*}
        \frac{X_{m/2}^2}{Y_{m/2} Z_{m/2}} = \frac{X_{m/2}}{Z_{m/2}} = \frac{d}{2} - \frac{1}{2t_d\left(\left[\frac{m}{2}\right]K\right)}.
    \end{equation*}
    Since $h(t_d([\frac{m}{2}]K)) = 0$, we have $-1/(2t_d([\frac{m}{2}]K))$ equals
    to $c_1, c_2$ for even $m$.
    By combining these equations,
    we obtain the desired expression of $d'$.
\end{proof}
An easy calculation using the following formulas
shows that this isogeny preserves the level structure:
\begin{align*}
    (X:Y:Z) + P^\hess_d &= (\omega^2 X : \omega Y : Z),\\
    (X:Y:Z) + Q^\hess_d &= (Z : X :Y).
\end{align*}

By the above results,
we obtain the following proposition.
\begin{proposition}
    The set of pairs
    $\{(\EEparam{d}{\hess}, t_d)\}$
    is a good model of $\hess$.
\end{proposition}

From the theorems in Section \ref{subsec:main_theorem},
we obtain the following corollary.
\begin{corollary}
    For any positive integer $m$ prime to $3$,
    there exists the modular polynomial $\Phi_m^{\hess}(X, Y) \in \Z[X, Y]$
    of order $m$ for $\hess$.
    In addition,
    if $m > 1$, then $\Phi_m^{\hess}(X, Y)$ is symmetric.
\end{corollary}

As in the case of Montgomery curves,
we can show some congruence properties of the coefficients
of our modular polynomials for Hessian curves.
\begin{proposition}\label{prop:coeff_zero_hessian}
    Let $m > 1$ be an integer prime to $3$.
    The coefficient of $X^i Y^j$ of $\Phi_m^{\hess}(X, Y)$ is zero
    unless $i + j \equiv \psi(m) \pmod{3}$ if $m \equiv 1 \pmod{3}$,
    and unless $i - j \equiv 0 \pmod{3}$ if $m \equiv 2 \pmod{3}$.
\end{proposition}
\begin{proof}
    Let $d \in \image\hess$.
    Then, the isomorphism $\iota_d : E_d^\hess \to E_{\omega d}^\hess$
    defined by $(X : Y : Z) \mapsto (X : Y : \omega^{-1}Z)$
    sends $(P_d^\hess, Q_d^\hess)$ to $(P_{\omega d}^\hess, P_{\omega d}^\hess + Q_{\omega d}^\hess)$.

    A similar argument to the proof for Proposition \ref{prop:coeff_zero_montgomery}
    using this isomorphism
    shows that
    \begin{equation*}
        \Phi_m^{\hess}(X, Y) = \begin{cases}
            \omega^{-\psi(m)}\Phi_m^{\hess}(\omega X, \omega Y) & \text{if } m \equiv 1 \pmod{3},\\
            \Phi_m^{\hess}(\omega X, \omega^{-1} Y) & \text{if } m \equiv 2 \pmod{3}.
            \end{cases}
    \end{equation*}
    This implies the desired result.
\end{proof}

\section{Experimental results}\label{sec:experiment}
We implemented the algorithm described in Section \ref{sec:algorithms}
to compute the modular polynomials of prime order $\ell$ for Montgomery curves and Hessian curves
by using SageMath~\cite{sagemath}.
Our code is based on and partially reuses the implementation~\cite{KR2024implementation}
by Kunzweiler and Robert.
For the implementation,
we use $B_\ell$ for the bound of the logarithmic height of
the modular polynomial of order $\ell$, instead of
the bound in Conjecture \ref{conj:height_bound}.

Our implementation is available at
\url{https://github.com/Yoshizumi-Ryo/modular_polynomial_level_sage}.
We computed the modular polynomials for Montgomery curves
and Hessian curves of prime order $\ell$ for all primes less than $150$.
The results are stored in the release page of this repository.
We confirmed that
the computed polynomials satisfy the conditions for isogenies
by computing isogenies on finite fields whose characteristic is
different from primes used in the CRT method.

\Cref{conj:height_bound} holds
for the computed modular polynomials
except for the case of Hessian curves with $\ell = 2, 5$.
The logarithmic height of $\Phi_2^{\hess}$ (resp. $\Phi_5^{\hess}$) is approximately $3.99$
(resp. $12.36$),
which is greater than $B_2/12 \approx 3.69$ (resp. $B_5/12 \approx 11.52$).

\Cref{conj:evaluation_at_non_image} holds for
the computed modular polynomials of Montgomery curves and Hessian curves.
In the case of Montgomery curves,
the permutation $\sigma$ is the identity.
In the case of Hessian curves,
the permutation $\sigma$ is given by
$3\omega^i \mapsto 3\omega^{\ell i}$
for $i = 0, 1, 2$.

\bibliography{
    bibtex/cryptobib/abbrev3,
    bibtex/cryptobib/crypto,
    bibtex/mybiblio,
    bibtex/yoshizumi_ref
}

\end{document}